\documentclass{article}
\usepackage[utf8]{inputenc}

\usepackage{amsmath,amssymb,amsthm}
\usepackage{bm}
\usepackage{algorithm,algorithmic}
\usepackage{enumerate}
\usepackage[margin=25mm]{geometry}
\usepackage[numbers]{natbib}
\usepackage{mathtools}
\mathtoolsset{showonlyrefs,showmanualtags}
\usepackage[titletoc,title]{appendix}

\newcommand{\argmax}{\mathop{\rm argmax}\limits}
\newcommand{\argmin}{\mathop{\rm argmin}\limits}
\newcommand{\dom}{\mathop{\rm dom}}
\newcommand{\prox}{\mathrm{prox}}
\newcommand{\dist}{\mathrm{dist}}
\newcommand{\diag}{\mathrm{diag}}
\newcommand{\sign}{\mathrm{sign}}

\newtheorem{theorem}{Theorem}[section]
\newtheorem{corollary}{Corollary}[section]
\newtheorem{lemma}{Lemma}[section]

\newtheorem{remark}{Remark}[section]

\newtheorem{proposition}{Proposition}[section]

\title{Proximal Diagonal Newton Methods for Composite Optimization Problems}
\author{Shotaro Yagishita\thanks{Department of Industrial and Systems Engineering, Chuo University, Japan, E-mail: a15.fjng@g.chuo-u.ac.jp}
\and
Shummin Nakayama\thanks{Info-Powered Energy System Research Center, The University of Electro-Communications, Japan, E-mail: snakayama@uec.ac.jp}
}
\date{\today}

\begin{document}

\maketitle

\begin{abstract}
This paper proposes new proximal Newton-type methods with a diagonal metric for solving composite optimization problems whose objective function is the sum of a twice continuously differentiable function and a proper closed directionally differentiable function.
Although proximal Newton-type methods using diagonal metrics have been shown to be superior to the proximal gradient method numerically, no theoretical results have been obtained to suggest this superiority.
Even though our proposed method is based on a simple idea, its convergence rate suggests an advantage over the proximal gradient method in certain situations.
Numerical experiments show that our proposed algorithms are effective, especially in the nonconvex case.
\end{abstract}

\section{Introduction}\label{sec:intro}
In this paper, algorithms solving the following optimization problems are considered:
\begin{align}\label{problem:general}
\underset{x\in\mathbb{R}^n}{\mbox{minimize}} & \quad F(x)\coloneqq f(x)+g(x),
\end{align}
where $f:\mathbb{R}^n\to\mathbb{R}$ is twice continuously differentiable, $g:\mathbb{R}^n\to(-\infty,\infty]$ is a proper closed directionally differentiable function, and $F$ is bounded below.
Since many problems in machine learning, signal processing, and statistical inference can be formulated as above, optimization problems of the form \eqref{problem:general} receive a lot of interest.
Proximal gradient methods (PGMs, also called forward-backward splitting methods) are popular iterative methods for solving composite optimization problems \citep{parikh2014proximal}.
When $g$ has a simple structure (e.g., $\ell_1$-norm), each iteration of the PGMs is executed efficiently.
While the usefulness of the PGMs is well known, convergence can be slow for ill-conditioned problems.
In the last decade, proximal Newton-type methods (PNMs, also called variable metric forward-backward splitting methods) have attracted attention as a possible solution to this issue \citep{becker2012quasi,lee2014proximal}.
The method is also known as the successive quadratic approximation method \citep{byrd2016inexact}.
By using an approximation of $\nabla^2f$ as a metric in each iteration, the number of iterations of the PNMs tends to be reduced compared to the PGMs. 
On the other hand, in general, solving the subproblem of the PNMs, that is, the scaled proximal mapping of $g$ cannot avoid iterative methods, even if $g$ is equipped with a simple structure.
Thus, the computational cost of PNMs per iteration is generally more expensive than that of the PGM.

Recently, PNMs with metrics restricted to diagonal matrices have been proposed \citep{park2020variable,bonettini2016variable}.
They only consider convex $g$'s, but the scaled proximal mapping can be easily computed for some nonconvex $g$.
Thus, there is no need to use an iterative method to compute the scaled proximal mapping for simple $g$.
Numerical experiments have reported advantages of the PNMs with diagonal metrics compared to the PGMs.
However, to the best of our knowledge, there is no theoretical result that implies the superiority of the PNMs with diagonal metrics over the PGMs.

In this paper, we propose new PNMs that use the diagonal part of the Hessian matrix of $f$ as the metric.
Although the proposed methods are simple, they have theoretical properties indicating superiority to the PGMs.
The theoretical results tell us for which problems our proposed methods are superior to the PGMs.
In addition, the scaled proximal mappings of some nonconvex functions are calculated explicitly.
Numerical experiments demonstrate the effectiveness of our algorithm, especially when $g$ is nonconvex.

The rest of this paper is organized as follows.
The remainder of this section is devoted to related works, notation, and preliminary results.
In the next section, we introduce our proposed algorithm, the proximal diagonal Newton method (PDNM), show some convergence results, and also explain the cases where the PDNM is superior to the PGMs.
Section \ref{sec:prox-calculus} is devoted to the calculation of the scaled proximal mapping with diagonal metrics of some nonconvex functions.
Some numerical experiments to demonstrate the advantage of PDNM are reported in Section \ref{sec:experiments}.
Finally, Section \ref{sec:conclusion} concludes the paper with some remarks.

\subsection{Related works}
Variable metric forward-backward splitting methods (i.e., PNMs) appeared earlier in \citep{chen1997convergence}.
In PNMs, how to solve the scaled proximal mapping, which appears in the subproblem, is the most crucial issue in practical use.
\citet{becker2019quasi} has shown that the scaled proximal mapping of a proper closed convex function can be reduced to an $r$-dimensional equation when the metric is of the form diagonal $\pm$ rank-$r$, and proposed a method for solving the equation.
The subproblem of the PNMs is solved exactly by performing the bisection method for a finite number of points when $r=1$, $g$ is proper closed convex and separable, and the proximal mapping of each univariate function is piecewise affine \citep{becker2012quasi,karimi2017imro}.
In each iteration of general PNMs for convex $g$, the scaled proximal mappings are calculated inexactly by using a numerical method (see, e.g., \citep{byrd2016inexact,scheinberg2016practical,li2017inexact,lee2019inexact,nakayama2021inexact,kanzow2021globalized}).
Recently, for difference-of-convex (DC) $g$, inexact PNMs combined with DC decomposition have been proposed \citep{liu2022inexact,nakayama2023inexact}.

The use of a diagonal matrix as an approximation of the Hessian matrix of $f$ appeared in \citep{gill1979conjugate,gilbert1989some} for the smooth unconstrained problem (i.e., $g=0$).
\citet{zhu1999quasi} proposed update rules for the metrics based on a certain variational principle.
Recently, superlinear convergence of a Newton-type method with an inverse diagonal BFGS update when $f$ is separable and strongly convex is established \citep{li2022diagonal}.
Diagonal metrics are also used in a first order primal-dual algorithm \citep{pock2011diagonal}, the Douglas-Rachford splitting \citep{giselsson2014diagonal}, subgradient methods \citep{duchi2011adaptive}, etc. other than the PNMs.

\subsection{Notation and Preliminaries}
For an integer $n$, the set $[n]$ is defined by $[n]\coloneqq\{1,\ldots,n\}$.
The largest eigenvalue and smallest eigenvalue of a square matrix $A$ are denoted by $\lambda_{\max}(A)$ and $\lambda_{\min}(A)$, respectively.
For symmetric matrices $A$ and $B$, $A\succeq B$ ($A\succ B$) means that $A-B$ is a positive semidefinite (definite) matrix.
Let $I$ and $e_i$ denote the identity matrix of appropriate size and the vector whose $i$-th element is $1$ and 0 otherwise, respectively.
For a matrix $A$, by $\diag(A)$ we denote the diagonal matrix whose diagonal elements coincide with those of $A$.
The standard inner product of $x,y\in\mathbb{R}^n$ is denoted by $x^\top y$.
The $\ell_1$ norm and the $\ell_2$ norm of $x\in\mathbb{R}^n$ are defined by $\|x\|_1\coloneqq\sum_{i\in[n]}|x_i|$ and $\|x\|_2\coloneqq\sqrt{x^\top x}$, respectively.
We define $\|x\|_A\coloneqq\sqrt{x^\top Ax}$ for a positive definite matrix $A$ and $x\in\mathbb{R}^n$.
For a matrix $A$, $\|A\|_{2,2}$ denotes the operator norm that is defined by $\|A\|_{2,2}\coloneqq\max\{\|Ax\|_2\mid\|x\|_2\le1\}$.
We denote $\max\{\xi,0\}$ by $(\xi)_+$.
Let $\dist(x,\mathcal{X})$ denote the distance between $x\in\mathbb{R}^n$ and $\mathcal{X}\subset\mathbb{R}^n$, that is, $\dist(x,\mathcal{X})\coloneqq\inf\{\|x-y\|_2\mid y\in\mathcal{X}\}$.

The domain of $g$ is denoted by $\dom g$.
For $x\in\dom g$, $\mathcal{F}_g(x)$ is the feasible cone of $\dom g$ at $x$, that is, $\mathcal{F}_g(x)\coloneqq\{d\in\mathbb{R}^n\mid x+\eta d\in\dom g,~ \eta\in(0,\delta)~ \mbox{for some}~ \delta>0\}$.
The function $g$ is directionally differentiable if the directional derivative of $g$ at $x$ in direction $d$ exists for any $x\in\dom g$ and $d\in\mathcal{F}_g(x)$, that is,
\begin{align}
g'(x;d)\coloneqq\lim_{\eta\searrow 0}\frac{g(x+\eta d)-g(x)}{\eta}\in[-\infty,\infty].
\end{align}
We call $x^*\in\dom g$ a d(irectional)-stationary point of \eqref{problem:general} if $F'(x^*;d)=\nabla f(x^*)^\top d+g'(x^*;d)\ge0$ holds for all $d\in\mathcal{F}_g(x^*)$.
For given $\varepsilon>0$, a point $x^*\in\dom g$ is said to be a $\varepsilon$-stationary point (in the sense of d-stationarity) if $F'(x^*;d)\ge-\varepsilon\|d\|_2$ holds for all $d\in\mathcal{F}_g(x^*)$.

For a positive definite matrix $A$, the scaled proximal mapping of $g$ at $x$ is defined by
\begin{align}
    \prox_g^A(x)\coloneqq\argmin_{y}\Big\{g(y)+\frac{1}{2}\|y-x\|_A^2\Big\}.
\end{align}
If $\prox_g^A(x)$ is a singleton, we denote its element by the same symbol.
We call $G:\mathbb{R}^n\to\mathbb{R}^n$ ($H:\mathbb{R}^n\to\mathbb{R}^{n\times n}$) Lipschitz continuous if there exists $L\ge0$ such that
\begin{align}
\|\nabla G(x)-\nabla G(y)\|_2\le L\|x-y\|_2 \quad (\|H(x)-H(y)\|_{2,2}\le L\|x-y\|_2)
\end{align}
for any $x, y\in\mathbb{R}^n$.
A function $f$ is said to be $m$-strongly convex with respect to a norm $\|\cdot\|$ for $m>0$ if it holds that
\begin{align}
    f(\lambda x+(1-\lambda)y)\le\lambda f(x)+(1-\lambda)f(y)-\frac{m}{2}\lambda(1-\lambda)\|x-y\|^2
\end{align}
for any $x, y\in\mathbb{R}^n$ and $\lambda\in[0,1]$.
When the corresponding norm is the $\ell_2$ norm, we simply say that $f$ is $m$-strongly convex.

\begin{lemma}\label{lem:f-p-g-ineq}
Suppose that $g$ is convex.
Let $\beta>0,~ x,y\in\mathbb{R}^n$, $H$ be a positive definite matrix, and $x^+\in\prox_g^H(x-H^{-1}\nabla f(x))$.
If it holds that
\begin{align}\label{eq:p-line-search}
f(x^+)\le f(x)+\nabla f(x)^\top(x^+-x)+\frac{\beta}{2}\|x^+-x\|_{H}^2,
\end{align}
then we have
\begin{align}
    F(y)-F(x^+)\ge\frac{1}{2}\|y-x^+\|_{H}^2-\frac{1}{2}\|y-x\|_{H}^2+\ell_f(y,x)-\frac{\beta-1}{2}\|x^+-x\|_{H}^2,
\end{align}
where $\ell_f(y,x)\coloneqq f(y)-f(x)-\nabla f(x)^\top(y-x)$.
\end{lemma}

\begin{proof}
Since the function $g+\frac{1}{2}\|\cdot-(x-H^{-1}\nabla f(x))\|_{H}^2$ is $1$-strongly convex with respect to $\|\cdot\|_H$, it follow from \citet[Theorem 5.25]{beck2017first} that
\begin{align}
    g(y)+\frac{1}{2}\|y-(x-H^{-1}\nabla f(x))\|_{H}^2-g(x^+)-\frac{1}{2}\|x^+-(x-H^{-1}\nabla f(x))\|_{H}^2\ge\frac{1}{2}\|y-x^+\|_{H}^2
\end{align}
Combining this with \eqref{eq:p-line-search}, we have
\begin{align}
    &g(y)-g(x^+)\\
    &\ge\frac{1}{2}\|y-x^+\|_{H}^2-\frac{1}{2}\|y-x\|_{H}^2-\nabla f(x)^\top(y-x)+\frac{1}{2}\|x^+-x\|_{H}^2+\nabla f(x)^\top(x^+-x)\\
    &\ge\frac{1}{2}\|y-x^+\|_{H}^2-\frac{1}{2}\|y-x\|_{H}^2+\frac{1}{2}\|x^+-x\|_{H}^2-\nabla f(x)^\top(y-x)+f(x^+)-f(x)-\frac{\beta}{2}\|x^+-x\|_{H}^2.
\end{align}
We obtain the desired result by adding $f(y)-f(x^+)$ to both sides.
\end{proof}

This paper uses the following quantities in the convergence analysis:
\begin{align}
    \sigma\coloneqq\inf_{x,v\neq0,0\le\lambda\le1}\frac{v^\top\nabla^2f(x+\lambda v)v}{v^\top\diag(\nabla^2f(x))v},\quad\tau\coloneqq\sup_{x,v\neq0,0\le\lambda\le1}\frac{v^\top\nabla^2f(x+\lambda v)v}{v^\top\diag(\nabla^2f(x))v}.
\end{align}
The above quantities $\sigma, \tau$ are well-defined when $\diag(\nabla^2f(x))\succ O$ for any $x\in\mathbb{R}^n$.
Clearly, $\sigma\le\tau$.
If $\nabla f$ is Lipschitz continuous with constant $L_1$ and there exists $m>0$ such that $\diag(\nabla^2f(x))\succeq mI$ for any $x\in\mathbb{R}^n$, since $\|\nabla^2f(x+\lambda v)\|_{2,2}\le L_1$ \citep[Theorem 5.12]{beck2017first}, it holds that
\begin{align}
    v^\top\nabla^2f(x+\lambda v)v &\le\lambda_{\max}(\nabla^2f(x+\lambda v))\|v\|_2^2\\
    &\le\|\nabla^2f(x+\lambda v)\|_{2,2}\|v\|_2^2\\
    &\le L_1\|v\|_2^2\\
    &\le\frac{L_1}{m}v^\top\diag(\nabla^2f(x))v,
\end{align}
which implies that 
\begin{align}\label{eq:ineq-tau}
    0\le\tau\le\frac{L_1}{m}<\infty.
\end{align}
Let us assume, in addition, that $f$ is m-strongly convex.
Since $\nabla f$ is Lipschitz continuous with constant $L_1$ if and only if $\|\nabla^2f(x)\|_{2,2}\le L_1$ for any $x\in\mathbb{R}^n$ \citep[Theorem 5.12]{beck2017first}, it holds that
\begin{align}\label{eq:boundedness-diag}
    (\nabla^2f(x))_{i,i}\le\|\nabla^2f(x)e_i\|_2\le L_1\|e_i\|_2=L_1,
\end{align}
which yields
\begin{align}
    v^\top\diag(\nabla^2f(x))v\le L_1\|v\|_2^2\le\frac{L_1}{m}v^\top\nabla^2f(x+\lambda v)v.
\end{align}
As a result, we obtain that
\begin{align}\label{eq:ineq-sigma}
    \sigma\ge\frac{m}{L_1}>0.
\end{align}

\section{Proximal Diagonal Newton Method}\label{sec:PDNM}
This section introduces PNMs with a new diagonal metric and investigates their convergence.
Letting $x^0\in\dom g, \eta>1, \beta>0$, our first proposed method (Algorithm \ref{alg:PDNM}) repeats
\begin{align}\label{eq:prox-grad-PDNM}
    x^{t+1}\in\prox_g^{\eta^kD_t}\left(x^t-(\eta^kD_t)^{-1}\nabla f(x^t)\right)
\end{align}
with $k=k_t$ and $D_t=\diag(\nabla^2f(x^t))\succ O$, where $k_t$ is the smallest nonnegative integer $k$ that satisfies
\begin{align}\label{eq:line-search}
    f(x^{t+1})\le f(x^t)+\nabla f(x^t)^\top(x^{t+1}-x^t)+\frac{\beta}{2}\|x^{t+1}-x^t\|_{\eta^kD_t}^2.
\end{align}
We refer to it as the proximal diagonal Newton method (PDNM).
The bottleneck of \eqref{eq:prox-grad-PDNM} is the diagonal scaled proximal mapping of $g$, which is easily performed if $g$ has a simple structure, even when it is nonconvex.
The scaled proximal mappings with diagonal metrics of several convex functions have been given by \citet{park2019variable}.
Examples of the diagonal scaled proximal mappings of nonconvex functions are provided in Section \ref{sec:prox-calculus}.

\begin{algorithm}[H]
\caption{Proximal diagonal Newton method for the problem \eqref{problem:general}}
    \label{alg:PDNM}
    \begin{algorithmic}
    \STATE {\bfseries Input:} $x^0\in\dom g,~ \eta>1,~ \beta>0$ and $t=0$.
    \REPEAT
    \STATE Calculate $H_t=\diag(\nabla^2f(x^t))$.
    \STATE Compute $x^{t+1}\in\prox_g^{H_t}\left(x^t-H_t^{-1}\nabla f(x^t)\right)$.
    \WHILE{$f(x^{t+1})>f(x^t)+\nabla f(x^t)^\top(x^{t+1}-x^t)+\frac{\beta}{2}\|x^{t+1}-x^t\|_{H_t}^2$}
    \STATE Set $H_t\leftarrow\eta H_t$.
    \STATE Compute $x^{t+1}\in\prox_g^{H_t}\left(x^t-H_t^{-1}\nabla f(x^t)\right)$.
    \ENDWHILE
    \STATE Set $t\leftarrow t+1$.
    \UNTIL Terminated criterion is satisfied.
  \end{algorithmic}
\end{algorithm}

Under the assumptions of Lipschitz continuity of the gradient of $f$ and uniformly positive definiteness of $\diag(\nabla^2f(x))$, each iteration of Algorithm \ref{alg:PDNM} is well-defined.
\begin{lemma}\label{lem:well-def-1}
Suppose that $\nabla f$ is Lipschitz continuous with constant $L_1$ and there exists $m>0$ such that $\diag(\nabla^2f(x))\succeq mI$.
Then, $k_t$ is well-defined and it holds that
\begin{align}\label{eq:boundedness-hessian-1}
    \eta^{k_t}\le\max\left\{1,\frac{\eta\tau}{\beta}\right\}, \quad mI\preceq\eta^{k_t}D_t\preceq\max\left\{1,\frac{\eta\tau}{\beta}\right\}L_1I
\end{align}
for any $t$, where $D_t=\diag(\nabla^2f(x^t))$.
\end{lemma}

\begin{proof}
From Taylor's theorem, there exists $\lambda\in(0,1)$ such that
\begin{align}
    f(x^{t+1})=f(x^t)+\nabla f(x^t)^\top(x^{t+1}-x^t)+\frac{1}{2}\|x^{t+1}-x^t\|_{\nabla^2f(x^t+\lambda(x^{t+1}-x^t))}^2.
\end{align}
If $k$ satisfies $\beta\eta^k\ge\tau$, it follow from the assumptions and definition of $\tau$ that
\begin{align}
    f(x^{t+1}) &\le f(x^t)+\nabla f(x^t)^\top(x^{t+1}-x^t)+\frac{\tau}{2}\|x^{t+1}-x^t\|_{D_t}^2\\
    &\le f(x^t)+\nabla f(x^t)^\top(x^{t+1}-x^t)+\frac{\beta}{2}\|x^{t+1}-x^t\|_{\eta^kD_t}^2.
\end{align}
Thus, $k_t$ is well-defined.
Obviously $\eta^{k_t}=1$ when $k_t=0$.
If $k_t>0$, then the acceptance criterion \eqref{eq:line-search} is not satisfied for $k=k_t-1$ and hence $\beta\eta^{k_t-1}<\tau$.
Consequently, we have
\begin{align}\label{eq:boundedness-power}
    \eta^{k_t}\le\max\left\{1,\frac{\eta\tau}{\beta}\right\}.
\end{align}
Combining this with \eqref{eq:boundedness-diag} yields
\begin{align}
    mI\preceq\eta^{k_t}D_t\preceq\max\left\{1,\frac{\eta\tau}{\beta}\right\}L_1I,
\end{align}
which is the desired result.
\end{proof}

The assumption of uniformly positive definiteness of $\diag(\nabla^2f(x))$ is not restrictive.
Actually, considering the sparse regression problem, that is, $f(x)=\frac{1}{2}\|b-Ax\|_2^2$ and $g$ is a function inducing sparsity such as the $\ell_1$ norm, although $f$ is not a strongly convex function in general, the diagonal components of $\nabla^2f$ are constant and positive since we can assume that each column of $A$ is not a zero-vector without loss of generality.

We first show the global convergence of the PDNM without convexity assumption.

\begin{theorem}\label{thm:global-convergence-1}
Suppose that all the assumptions of Lemma \ref{lem:well-def-1} are satisfied and $0<\beta<1$ and $\{x^t\}_{t=0}^\infty$ is generated by Algorithm \ref{alg:PDNM} without termination.
Then, $\|x^{t+1}-x^t\|_2$ converges to $0$ and any accumulation point of $\{x^t\}_{t=0}^\infty$ is a d-stationary point of \eqref{problem:general}.
Moreover, it holds that $\lim_{t\to\infty}\dist(x^t, \mathcal{X}^*)=0$ if $\{x^t\}_{t=0}^\infty$ is bounded, where $\mathcal{X}^*$ is the set of d-stationary points of \eqref{problem:general}.
If $g$ is convex, the above results also hold for $1\le\beta<2$.
\end{theorem}

\begin{proof}
It follows from the optimality of $x^{t+1}$ to the subproblem that
\begin{align}
    g(x^{t+1})+\frac{1}{2}\|x^{t+1}-x^t\|_{\eta^{k_t}D_t}^2+\nabla f(x^t)^\top(x^{t+1}-x^t)\le g(x^t),
\end{align}
where $D_t=\diag(\nabla^2f(x^t))$.
We obtain from the acceptance criterion \eqref{eq:line-search} and the above that
\begin{align}\label{eq:sufficient-descent-1}
    F(x^{t+1})\le F(x^t)-\frac{1-\beta}{2}\|x^{t+1}-x^t\|_{\eta^{k_t}D_t}^2.
\end{align}
Since \eqref{eq:boundedness-hessian-1} and \eqref{eq:sufficient-descent-1} imply that the assumptions of Lemma \ref{lem:global-convergence} (see Appendix \ref{sec:lemma}) are satisfied, we have the desired result.

If $g$ is convex, by using Lemma \ref{lem:f-p-g-ineq} with $x=x^t,~ y=x^t,~ H=\eta^{k_t}D_t$, we obtain
\begin{align}\label{eq:strong-sufficient-descent-1}
    F(x^{t+1})\le F(x^t)-\frac{2-\beta}{2}\|x^{t+1}-x^t\|_{\eta^{k_t}D_t}^2.
\end{align}
Thus, the convergence results hold for $0<\beta<2$.
\end{proof}

Under the assumptions of Theorem \ref{thm:global-convergence-1}, let us consider the following terminated criterion:
\begin{align}\label{eq:termination}
    \|\nabla f(x^t)-\nabla f(x^{t-1})+\eta^{k_{t-1}}D_{t-1}(x^{t-1}-x^t)\|_2\le\varepsilon
\end{align}
for some $\varepsilon>0$.
We see from the Lipschitz continuity of $\nabla f$ and boundedness of $\eta^{k_{t-1}}D_{t-1}$ that
\begin{align}
    &\|\nabla f(x^t)-\nabla f(x^{t-1})+\eta^{k_{t-1}}D_{t-1}(x^{t-1}-x^t)\|_2\\
    &\le\|\nabla f(x^t)-\nabla f(x^{t-1})\|_2+\|\eta^{k_{t-1}}D_{t-1}\|_{2,2}\|x^t-x^{t-1}\|_2\\
    &\le L_1\|x^t-x^{t-1}\|_2+\max\left\{1,\frac{\eta\tau}{\beta}\right\}L_1\|x^t-x^{t-1}\|_2\\
    &=\max\left\{2,1+\frac{\eta\tau}{\beta}\right\}L_1\|x^t-x^{t-1}\|_2.
\end{align}
From this and $\lim_{t\to\infty}\|x^t-x^{t-1}\|_2=0$, the termination criterion \eqref{eq:termination} must hold at some $t$.
In addition, for any $d\in\mathcal{F}_g(x^t)$, since it follows from the d-stationarity of $x^t$ to subproblem that
\begin{align}
    g'(x^t;d)+\left(\eta^{k_{t-1}}D_{t-1}(x^t-x^{t-1})+\nabla f(x^{t-1})\right)^\top d\ge0,
\end{align}
the termination criterion \eqref{eq:termination} implies that
\begin{align}
    F'(x^t;d) &=\nabla f(x^t)^\top d+g'(x^t;d)\\
    &\ge\left(\nabla f(x^t)-\nabla f(x^{t-1})+\eta^{k_{t-1}}D_{t-1}(x^{t-1}-x^t)\right)^\top d\\
    &\ge-\|\nabla f(x^t)-\nabla f(x^{t-1})+\eta^{k_{t-1}}D_{t-1}(x^{t-1}-x^t)\|_2\|d\|_2\\
    &\ge-\varepsilon\|d\|_2.
\end{align}
As a result, it is guaranteed that $\varepsilon$-stationary point of \eqref{problem:general} can be obtained by a finite number of iterations of Algorithm \ref{alg:PDNM} with termination criterion \eqref{eq:termination}.

In the following, $\mathcal{O}(1/t^p)$ sublinear convergence rate in the objective function values is established under the convexity assumption.

\begin{theorem}\label{thm:sub-linear-convergence-1}
Let $0<\beta\le1$ and $\{x^t\}_{t=0}^\infty$ be a sequence generated by Algorithm \ref{alg:PDNM} without termination.
We suppose that $f$ and $g$ are convex, $\nabla f$ is Lipschitz continuous with constant $L_1$, and there exist $m>0$ such that $\diag(\nabla^2f(x^t))\succeq mI$ and an optimal solution $x^*$ of \eqref{problem:general}.
Assume further that
\begin{align}\label{eq:hessian-gap}
    C\coloneqq\sup_{t\ge1}\frac{\sum_{s=0}^{t-2}\left(\|x^{s+1}-x^*\|_{D_{s+1}/M_{s+1}}^2-\|x^{s+1}-x^*\|_{D_s/M_s}^2\right)}{t^{1-p}}<\infty
\end{align}
for some $p\in(0,1]$, then it holds that
\begin{align}\label{eq:sub-linear-convergence-1}
    F(x^t)-F(x^*)\le\frac{M\left(\|x^0-x^*\|_{D_0/M_0}^2+C\right)}{2t^p}
\end{align}
for any $t\ge1$, where $D_t=\diag(\nabla^2f(x^t))$, $M_t=\max_{1\le i\le n}(\nabla^2f(x^t))_{i,i}$, and $M=\max\big\{1,\frac{\eta\tau}{\beta}\big\}L_1$.
\end{theorem}

\begin{proof}
Since $f$ and $g$ are convex, by using Lemma \ref{lem:f-p-g-ineq} with $x=x^s,~ y=x^*,~ H=\eta^{k_s}D_s$, we have
\begin{align}
    F(x^*)-F(x^{s+1}) &\ge\frac{1}{2}\|x^{s+1}-x^*\|_{\eta^{k_s}D_s}^2-\frac{1}{2}\|x^s-x^*\|_{\eta^{k_s}D_s}^2+\ell_f(x^*,x^s)-\frac{\beta-1}{2}\|x^{s+1}-x^s\|_{\eta^{k_s}D_s}^2\\
    &\ge\frac{1}{2}\|x^{s+1}-x^*\|_{\eta^{k_s}D_s}^2-\frac{1}{2}\|x^s-x^*\|_{\eta^{k_s}D_s}^2.
\end{align}
Dividing the above by $\eta^{k_s}M_s$ and then summing it over $s=0,\ldots,t-1$, we obtain
\begin{align}
    &\sum_{s=0}^{t-1}\frac{F(x^{s+1})-F(x^*)}{\eta^{k_s}M_s}\\
    &\le\frac{1}{2}\sum_{s=0}^{t-1}\left(\|x^s-x^*\|_{D_s/M_s}^2-\|x^{s+1}-x^*\|_{D_s/M_s}^2\right)\\
    &=\frac{1}{2}\|x^0-x^*\|_{D_0/M_0}^2+\frac{1}{2}\sum_{s=0}^{t-2}\left(\|x^{s+1}-x^*\|_{D_{s+1}/M_{s+1}}^2-\|x^{s+1}-x^*\|_{D_s/M_s}^2\right)-\frac{1}{2}\|x^t-x^*\|_{D_{t-1}/M_{t-1}}^2\\
    &\le\frac{\|x^0-x^*\|_{D_0/M_0}^2+Ct^{1-p}}{2}\\
    &\le\frac{\left(\|x^0-x^*\|_{D_0/M_0}^2+C\right)t^{1-p}}{2}.
\end{align}
We see from \eqref{eq:strong-sufficient-descent-1} that $\{F(x^t)\}$ is non-increasing and from Lemma \ref{lem:well-def-1} that $\eta^{k_s}M_s\le M$, and hence it holds that
\begin{align}
    \frac{t}{M}\left(F(x^t)-F(x^*)\right) &=\sum_{s=0}^{t-1}\frac{F(x^t)-F(x^*)}{M}\\
    &\le\sum_{s=0}^{t-1}\frac{F(x^{s+1})-F(x^*)}{\eta^{k_s}M_s}.
\end{align}
Consequently, we have the desired result \eqref{eq:sub-linear-convergence-1}.
\end{proof}

Theorem \ref{thm:sub-linear-convergence-1} is a slight extension of Theorem 3 of \citet{scheinberg2016practical}.
When $f$ is a convex quadratic function such that diagonal components of $\nabla^2f$ are positive, the assumption \eqref{eq:hessian-gap} is guaranteed, then we obtain $\mathcal{O}(1/t)$ sublinear convergence rate.

\begin{corollary}\label{cor:sub-linear-convergence-2}
Let $0<\beta\le1$, $\{x^t\}_{t=0}^\infty$ be a sequence generated by Algorithm \ref{alg:PDNM} without termination, $f(x)=\frac{1}{2}x^\top Qx+l^\top x$, and $D=\diag(Q)$, where $Q$ is a positive semidefinite matrix.
Assume that $g$ is convex, $m\coloneqq\min_{1\le i\le n}D_{i,i}>0$, and there exists an optimal solution $x^*$ of \eqref{problem:general}.
Then, we have
\begin{align}\label{eq:sub-linear-convergence-2}
    F(x^t)-F(x^*)\le\frac{M\|x^0-x^*\|_{D'}^2}{2t}
\end{align}
for any $t\ge1$, where $M=\max\big\{1,\frac{\eta\tau}{\beta}\big\}L_1$ and $D'=D/\max_{1\le i\le n}D_{i,i}$.
\end{corollary}

\begin{proof}
Noting that $\diag(\nabla^2f(x^t))=D$, the assumption \eqref{eq:hessian-gap} holds with $C=0, p=1$.
From Theorem \ref{thm:sub-linear-convergence-1}, we obtain the desired result \eqref{eq:sub-linear-convergence-2}.
\end{proof}

The convergence results so far are not specific to the PDNM, and are also valid for PGMs.
However, under certain assumptions, the PDNM has a quadratic convergence property that does not hold true for PGMs.

\begin{theorem}\label{thm:quadratic-convergence-1}
Let $1<\beta<2$, $\{x^t\}_{t=0}^\infty$ be a sequence generated by Algorithm \ref{alg:PDNM} without termination, and $f(x)=\sum_{i=1}^n f_i(x_i)$, where $f_i:\mathbb{R}\to\mathbb{R}$.
We suppose that $f$ is $m$-strongly convex, $\nabla f$ and $\nabla^2 f$ are Lipschitz continuous with constant $L_1$ and $L_2$, and $g$ is convex.
Then, the sequence $\{x^t\}_{t=0}^\infty$ converges to the unique optimal solution of \eqref{problem:general}.
Furthermore, for sufficiently large $t$, $k_t=0$ holds and the sequence converges Q-quadratically.
\end{theorem}

\begin{proof}
It follows from the strong convexity of $f$ that $\diag(\nabla^2f(x^t))=\nabla^2f(x^t)\succeq mI$.
Noting that the inequality \eqref{eq:strong-sufficient-descent-1} holds and hence $F(x^t)\le F(x^0)<\infty$, we see from strong convexity of $F$ that the sequence $\{x^t\}_{t=0}^\infty$ is bounded.
From Theorem \ref{thm:global-convergence-1}, the sequence $\{x^t\}_{t=0}^\infty$ converges to the unique optimal solution of \eqref{problem:general}.

Let
\begin{align}
    y^t=\prox_g^{D_t}\left(x^t-D_t^{-1}\nabla f(x^t)\right),
\end{align}
where $D_t=\nabla^2f(x^t)$.
Since $\nabla^2f$ is Lipschitz continuous, it holds that
\begin{align}\label{eq:cubic-ineq}
\begin{split}
    f(y^t) &\le f(x^t)+\nabla f(x^t)^\top(y^t-x^t)+\frac{1}{2}\|y^t-x^t\|_{D_t}^2+\frac{L_2}{6}\|y^t-x^t\|_2^3\\
    &\le f(x^t)+\nabla f(x^t)^\top(y^t-x^t)+\frac{1}{2}\|y^t-x^t\|_{D_t}^2\left(1+\frac{L_2}{3m}\|y^t-x^t\|_2\right).
\end{split}
\end{align}
One can obtain
\begin{align}\label{eq:quadratic}
\begin{split}
    \|y^t-x^*\|_2 &\le\frac{1}{\sqrt{m}}\|y^t-x^*\|_{D_t}\\
    &=\frac{1}{\sqrt{m}}\left\|\prox_g^{D_t}\left(x^t-D_t^{-1}\nabla f(x^t)\right)-\prox_g^{D_t}\left(x^*-D_t^{-1}\nabla f(x^*)\right)\right\|_{D_t}\\
    &\le\frac{1}{\sqrt{m}}\left\|x^t-x^*-D_t^{-1}(\nabla f(x^t)-\nabla f(x^*))\right\|_{D_t}\\
    &=\frac{1}{\sqrt{m}}\left\|D_t(x^t-x^*)-(\nabla f(x^t)-\nabla f(x^*))\right\|_{D_t^{-1}}\\
    &\le\frac{1}{m}\left\|D_t(x^t-x^*)-(\nabla f(x^t)-\nabla f(x^*))\right\|_2\\
    &\le\frac{L_2}{2m}\left\|x^t-x^*\right\|_2^2,
\end{split}
\end{align}
as well as the proof of Theorem 3.4 of \citet{lee2014proximal}, where $x^*$ is the unique optimal solution of \eqref{problem:general}.
Since we see from the above inequality that
\begin{align}
    \|y^t-x^t\|_2 &\le\|x^t-x^*\|_2+\|y^t-x^*\|_2\\
    &\le\|x^t-x^*\|_2+\frac{L_2}{2m}\left\|x^t-x^*\right\|_2^2,
\end{align}
it holds that $\|y^t-x^t\|_2\to0$ and hence the inequality $1+\frac{L_2}{3m}\|y^t-x^t\|_2\le\beta$ holds for sufficient large $t$.
Therefore, for sufficient large $t$, it holds that $x^{t+1}=y^t$, namely, $k_t=0$.
This and \eqref{eq:quadratic} imply quadratic convergence of $\{x^t\}_{t=0}^\infty$.
\end{proof}

The diagonal BFGS proposed by \citet{li2022diagonal} for smooth problems (i.e., $g=0$) has a local superlinear convergence property under a similar complete separability assumption.
The diagonal Newton metric is theoretically superior to the diagonal BFGS because it has the global quadratic convergence property.
While Theorem \ref{thm:quadratic-convergence-1} suggests an advantage of the PDNM, the assumption of complete separability, which implies that $\nabla^2f(x)$ is diagonal, is quite strong in practical settings.
Intuitively, there seems to be an advantage to the PDNM even when $\nabla^2f(x)$ is close to a diagonal matrix.
In fact, a result supporting this intuition can be obtained by analyzing an exponent of the following R-linear convergence rate.

\begin{theorem}\label{thm:r-linear-convergence}
Let $0<\beta\le1$, $\{x^t\}_{t=0}^\infty$ be a sequence generated by Algorithm \ref{alg:PDNM} without termination.
We suppose that $f$ is $m$-strongly convex, $\nabla f$ is Lipschitz continuous with constant $L_1$, and $g$ is convex.
Then, it holds that
\begin{align}
    \frac{m}{2}\|x^t-x^*\|_2^2\le F(x^t)-F(x^*)\le\left(1-\frac{\sigma}{\overline{\eta}\tau}\right)^t(F(x^0)-F(x^*))
\end{align}
for all $t\ge0$, where $\overline{\eta}=\max\big\{\frac{1}{\tau},\frac{\eta}{\beta}\big\}\ge1$.
\end{theorem}

\begin{proof}
Let $D_t\coloneqq\diag(\nabla^2f(x^t))$, $\lambda\coloneqq\frac{\sigma}{\overline{\eta}\tau}$, and $x_\lambda\coloneqq\lambda x^*+(1-\lambda)x^t$.
Note that $0<\frac{\sigma}{\overline{\eta}\tau}<1$ and it holds that
\begin{align}\label{eq:diagonal-strong-descent-lemma}
    \ell_f(x_\lambda,x^t)=\frac{1}{2}\|x_\lambda-x^t\|_{\nabla^2f(x^t+\lambda'(x_\lambda-x^t))}^2\ge\frac{\sigma}{2}\|x_\lambda-x^t\|_{D_t}^2,
\end{align}
where $\lambda'\in(0,1)$, the equality follows from Taylor's theorem, and the inequality from the definition of $\sigma$.
We obtain from \eqref{eq:diagonal-strong-descent-lemma} and Lemma \ref{lem:f-p-g-ineq} with $x=x^t,~ y=x_\lambda\coloneqq\lambda x^*+(1-\lambda)x^t,~ H=\eta^{k_t}D_t$ that
\begin{align}\label{eq:f-p-g-i-strong-0}
\begin{split}
    F(x_\lambda)-F(x^{t+1}) &\ge\frac{1}{2}\|x_\lambda-x^{t+1}\|_{\eta^{k_t}D_t}^2-\frac{1}{2}\|x_\lambda-x^t\|_{\eta^{k_t}D_t}^2+\ell_f(x_\lambda,x^t)-\frac{\beta-1}{2}\|x^{t+1}-x^t\|_{\eta^{k_t}D_t}^2\\
    &\ge-\frac{1}{2}\|x_\lambda-x^t\|_{\eta^{k_t}D_t}^2+\frac{\sigma}{2}\|x_\lambda-x^t\|_{D_t}^2\\
    &=-\frac{\lambda^2\eta^{k_t}}{2}\|x^*-x^t\|_{D_t}^2+\frac{\lambda^2\sigma}{2}\|x^*-x^t\|_{D_t}^2.
\end{split}
\end{align}
On the other hand, we have
\begin{align}
    f(x_\lambda) &=\lambda\left\{f(x^*)-\nabla f(x_\lambda)^\top(x^*-x_\lambda)-\frac{1}{2}\|x^*-x_\lambda\|_{\nabla^2f(x_\lambda+\lambda_1(x^*-x_\lambda))}^2\right\}\\
    &\quad+(1-\lambda)\left\{f(x^t)-\nabla f(x_\lambda)^\top(x^t-x_\lambda)-\frac{1}{2}\|x^t-x_\lambda\|_{\nabla^2f(x_\lambda+\lambda_2(x^t-x_\lambda))}^2\right\}\\
    &=\lambda f(x^*)-\frac{\lambda(1-\lambda)^2}{2}\|x^t-x^*\|_{\nabla^2f(x^t+\{\lambda+\lambda_1(1-\lambda)\}(x^*-x^t))}^2\\
    &\quad+(1-\lambda)f(x^t)-\frac{\lambda^2(1-\lambda)}{2}\|x^t-x^*\|_{\nabla^2f(x^t+\lambda(1-\lambda_2)(x^*-x^t))}^2\\
    &\le\lambda f(x^*)+(1-\lambda)f(x^t)-\frac{\lambda(1-\lambda)\sigma}{2}\|x^t-x^*\|_{D_t}^2,
\end{align}
where $\lambda_1,\lambda_2\in(0,1)$, the first equality follows from Taylor's theorem, and the inequality from the definition of $\sigma$.
Combining this with the convexity of $g$, the inequality \eqref{eq:f-p-g-i-strong-0}, and Lemma \ref{lem:well-def-1} yields 
\begin{align}
    F(x^{t+1}) &\le F(\lambda x^*+(1-\lambda)x^t)+\frac{\lambda^2\eta^{k_t}}{2}\|x^*-x^t\|_{D_t}^2-\frac{\lambda^2\sigma}{2}\|x^*-x^t\|_{D_t}^2\\
    &\le\lambda F(x^*)+(1-\lambda)F(x^t)-\frac{\lambda(1-\lambda)\sigma}{2}\|x^t-x^*\|_{D_t}^2+\frac{\lambda^2\overline{\eta}\tau}{2}\|x^*-x^t\|_{D_t}^2-\frac{\lambda^2\sigma}{2}\|x^*-x^t\|_{D_t}^2\\
    &=\lambda F(x^*)+(1-\lambda)F(x^t),
\end{align}
which implies that
\begin{align}
    F(x^{t+1})-F(x^*)\le(1-\lambda)(F(x^t)-F(x^*)).
\end{align}
As a result, since $F$ is $m$-strongly convex, it follow from \citet[Theorem 5.25]{beck2017first} that
\begin{align}
    \frac{m}{2}\|x^t-x^*\|_2^2\le F(x^t)-F(x^*)\le(1-\lambda)^t(F(x^0)-F(x^*)),
\end{align}
which is the desired result.
\end{proof}

To obtain the implication of Theorem \ref{thm:r-linear-convergence} from the exponent, we consider the strongly convex quadratic function of the form $f(x)=\frac{1}{2}x^\top Qx+l^\top x$.
Then $\sigma$ and $\tau$ are represented as
\begin{align}
    \sigma=\inf_{v\neq0}\frac{v^\top Qv}{v^\top Dv},\quad\tau=\sup_{v\neq0}\frac{v^\top Qv}{v^\top Dv},
\end{align}
where $D\coloneqq\diag(Q)$, that is, $\sigma$ and $\tau$ are the minimum and maximum generalized eigenvalues of $Q$ and $D$, respectively.
With this in mind, let us analyze the exponent of the above R-linear convergence rate.
Since $\overline{\eta}$ is a quantity derived from backtracking, we ignore it and consider only $\frac{\sigma}{\tau}$, which is greater than 0 and less than or equal to 1.
It is obvious that $\frac{\sigma}{\tau}=1$ if and only if $Q=\sigma D(=\tau D)$.
This is further equivalent to $Q=D$ because the diagonal components of $D$ and $Q$ coincide.
In other words, we can say that $\frac{\sigma}{\tau}$ is a quantity that becomes larger as $Q$ is close to a diagonal matrix, implying that convergence is faster then.

\begin{figure}[t]
\centering
\includegraphics[width=40em]{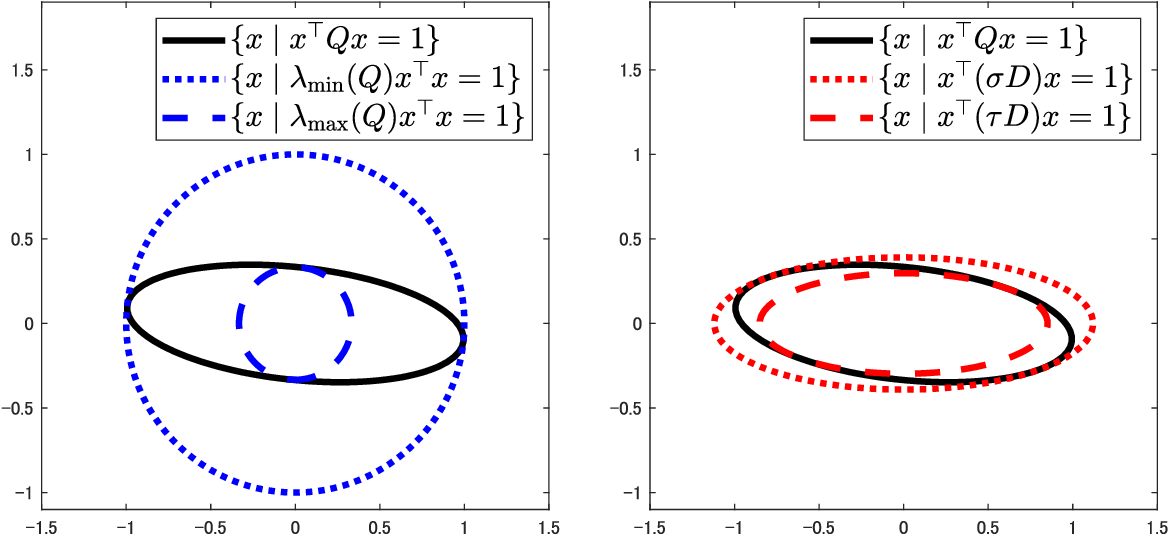}
\caption{The inscribed or circumscribed circles and inscribed or circumscribed ellipses of an ellipsoid}
\label{fig:circle-ellipsoid}
\end{figure}

Next, we attempt to understand a geometric meaning of $\frac{\sigma}{\tau}$.
The black solid line in Figure \ref{fig:circle-ellipsoid} is an ellipsoid $\{x\mid x^\top Qx=1\}$ related to a two-dimensional positive definite matrix $Q$.
From the definition of $\sigma$ and $\tau$, $\{x\mid x^\top(\sigma D)x=1\}$ and $\{x\mid x^\top(\tau D)x=1\}$ are the circumscribed and inscribed ellipsoids of $\{x\mid x^\top Qx=1\}$ (the right-hand side of Figure \ref{fig:circle-ellipsoid}), respectively.
We see from this observation that $\frac{\sigma}{\tau}$ is the square of the ratio of the distance from the center of the inscribed ellipsoid to the circumscribed ellipsoid.

Here we compare it with an exponent of a linear convergence rate of the PGMs.
The reciprocal of the condition number of $Q$ appearing in place of $\frac{\sigma}{\tau}$ (see, e.g., \citep[Theorem 10.29]{beck2017first}) is the square of the ratio of the radius of the inscribed circle to the circumscribed circle of $\{x\mid x^\top Qx=1\}$ (the left-hand side of Figure \ref{fig:circle-ellipsoid}).
Thus, it is implied that the PDNM is superior to the PGMs when as in Figure \ref{fig:circle-ellipsoid}, $Q$ is nearly diagonal and far from scalar multiples of the identity matrix.
While this implication is quite consistent with intuition, to the best of our knowledge, this is the first theoretical result that suggests an advantage of the diagonal metric even when the Hessian of $f$ is not diagonal.

Note that Q-linear convergence is also obtained in the case where $f$ is a strongly convex quadratic function.

\begin{theorem}\label{thm:q-linear-convergence}
Let $0<\beta\le1$, $\{x^t\}_{t=0}^\infty$ be a sequence generated by Algorithm \ref{alg:PDNM} without termination, $f(x)=\frac{1}{2}x^\top Qx+l^\top x$, and $D=\diag(Q)$, where $Q$ is a positive definite matrix.
Assume that $g$ is convex, and let $x^*$ be the unique optimal solution of \eqref{problem:general}.
Then, we have
\begin{align}
    \|x^{t+1}-x^*\|_D^2 &\le\left(1-\frac{\sigma}{\overline{\eta}\tau}\right)\|x^t-x^*\|_D^2,\\
    \|x^t-x^*\|_D^2 &\le\left(1-\frac{\sigma}{\overline{\eta}\tau}\right)^t\|x^0-x^*\|_D^2, \label{eq:linear-sequence}\\
    F(x^{t+1})-F(x^*) &\le\frac{\overline{\eta}\tau}{2}\left(1-\frac{\sigma}{\overline{\eta}\tau}\right)^{t+1}\|x^0-x^*\|_D^2. \label{eq:linear-objective}
\end{align}
for any $t\ge0$, where $\overline{\eta}=\max\big\{\frac{1}{\tau},\frac{\eta}{\beta}\big\}$.
\end{theorem}

\begin{proof}
Since it holds that
\begin{align}
    \ell_f(x^*,x^t)\geq\frac{\sigma}{2}\|x^*-x^t\|_{D}^2,
\end{align}
for the same reason as \eqref{eq:diagonal-strong-descent-lemma}, we obtain from this and Lemma \ref{lem:f-p-g-ineq} with $x=x^t,~ y=x^*,~ H=\eta^{k_t}D$ that
\begin{align}\label{eq:f-p-g-i-strong}
    F(x^*)-F(x^{t+1}) &\ge\frac{1}{2}\|x^*-x^{t+1}\|_{\eta^{k_t}D}^2-\frac{1}{2}\|x^*-x^t\|_{\eta^{k_t}D}^2+\ell_f(x^*,x^t)-\frac{\beta-1}{2}\|x^{t+1}-x^t\|_{\eta^{k_t}D}^2\\
    &\ge\frac{1}{2}\|x^*-x^{t+1}\|_{\eta^{k_t}D}^2-\frac{1}{2}\|x^*-x^t\|_{\eta^{k_t}D}^2+\frac{\sigma}{2}\|x^*-x^t\|_{D}^2.
\end{align}
Since $F(x^*)-F(x^{t+1})\le0$, we have
\begin{align}
    \|x^{t+1}-x^*\|_{D}^2\le\left(1-\frac{\sigma}{\eta^{k_t}}\right)\|x^t-x^*\|_D^2\le\left(1-\frac{\sigma}{\overline{\eta}\tau}\right)\|x^t-x^*\|_D^2,
\end{align}
where the second inequality follows from Lemma \ref{lem:well-def-1}.
Repeated use of this inequality yields \eqref{eq:linear-sequence}.
We also obtain from \eqref{eq:f-p-g-i-strong} that
\begin{align}
    F(x^{t+1})-F(x^*) &\le\frac{\eta^{k_t}-\sigma}{2}\|x^t-x^*\|_{D}^2-\frac{1}{2}\|x^{t+1}-x^*\|_{\eta^{k_t}D}^2\\
    &\le\frac{\overline{\eta}\tau-\sigma}{2}\|x^t-x^*\|_{D}^2\\
    &\le\frac{\overline{\eta}\tau-\sigma}{2}\left(1-\frac{\sigma}{\overline{\eta}\tau}\right)^t\|x^0-x^*\|_D^2\\
    &=\frac{\overline{\eta}\tau}{2}\left(1-\frac{\sigma}{\overline{\eta}\tau}\right)^{t+1}\|x^0-x^*\|_D^2.
\end{align}
This completes the proof.
\end{proof}

\subsection{PDNM with Nonmonotone Line Search }
To improve the practical performance of the PDNM, we consider incorporating a nonmonotone line search technique \citep{grippo1986nonmonotone,grippo2002nonmonotone}.
Concretely, the acceptance criterion \eqref{eq:line-search} of the PDNM is replaced by
\begin{align}\label{eq:nonmonotone-line-search}
    F(x^{t+1})\le\max_{(t-M+1)_+\le s\le t} F(x^s)-\frac{\alpha}{2}\|x^{t+1}-x^t\|_{\eta^kD_t}^2,
\end{align}
where $\eta>1,~ 0<\alpha<1,~ M\ge1$, and $D_t=\diag(\nabla^2f(x^t))$.
We refer to it as the nonmonotone proximal diagonal Newton method (NPDNM), which is shown in Algorithm \ref{alg:NPDNM}.

\begin{algorithm}[H]
\caption{Nonmonotone proximal diagonal Newton method for the problem \eqref{problem:general}}
    \label{alg:NPDNM}
    \begin{algorithmic}
    \STATE {\bfseries Input:} $x^0\in\dom g,~ \eta>1,~ 0<\alpha<1,~ M\ge1$, and $t=0$.
    \REPEAT
    \STATE Calculate $H_t=\diag(\nabla^2f(x^t))$.
    \STATE Compute $x^{t+1}\in\prox_g^{H_t}\left(x^t-H_t^{-1}\nabla f(x^t)\right)$.
    \WHILE{$F(x^{t+1})>\max_{(t-M+1)_+\le s\le t} F(x^s)-\frac{\alpha}{2}\|x^{t+1}-x^t\|_{H_t}^2$}
    \STATE Set $H_t\leftarrow\eta H_t$.
    \STATE Compute $x^{t+1}\in\prox_g^{H_t}\left(x^t-H_t^{-1}\nabla f(x^t)\right)$.
    \ENDWHILE
    \STATE Set $t\leftarrow t+1$.
    \UNTIL Terminated criterion is satisfied.
    \end{algorithmic}
\end{algorithm}

Each iteration of Algorithm \ref{alg:NPDNM} is well-defined under the same assumptions as in Algorithm \ref{alg:PDNM}.

\begin{lemma}\label{lem:well-def-2}
Let $k_t$ be the smallest nonnegative integer $k$ that satisfies \eqref{eq:nonmonotone-line-search}.
Suppose that $\nabla f$ is Lipschitz continuous with constant $L_1$ and there exists $m>0$ such that $\diag(\nabla^2f(x^t))\succeq mI$.
Then, $k_t$ is well-defined and it holds that $\eta^{k_t}\le\max\big\{1,\frac{\eta\tau}{(1-\alpha)}\big\}$ for any $t$.
\end{lemma}

\begin{proof}
Note that, as in Lemma \ref{lem:well-def-1}, it holds that
\begin{align}
    f(x^{t+1})\le f(x^t)+\nabla f(x^t)^\top(x^{t+1}-x^t)+\frac{\tau}{2}\|x^{t+1}-x^t\|_{D_t}^2.
\end{align}
The optimality of $x^{t+1}$ to the subproblem implies that
\begin{align}
    g(x^{t+1})+\frac{1}{2}\|x^{t+1}-x^t\|_{\eta^kD_t}^2+\nabla f(x^t)^\top(x^{t+1}-x^t)\le g(x^t),
\end{align}
where $D_t=\diag(\nabla^2f(x^t))$.
If $(1-\alpha)\eta^k\ge\tau$, then it follows from $D_t\succeq mI$ that
\begin{align}
    F(x^{t+1}) &\le F(x^t)-\frac{1}{2}\|x^{t+1}-x^t\|_{\eta^kD_t}^2+\frac{\tau}{2}\|x^{t+1}-x^t\|_{D_t}^2\\
    &=F(x^t)-\frac{\alpha}{2}\|x^{t+1}-x^t\|_{\eta^kD_t}^2-\frac{(1-\alpha)\eta^k}{2}\|x^{t+1}-x^t\|_{D_t}^2+\frac{\tau}{2}\|x^{t+1}-x^t\|_{D_t}^2\\
    &\le F(x^t)-\frac{\alpha}{2}\|x^{t+1}-x^t\|_{\eta^kD_t}^2\\
    &\le \max_{(t-M+1)_+\le s\le t} F(x^s)-\frac{\alpha}{2}\|x^{t+1}-x^t\|_{\eta^kD_t}^2.
\end{align}
Thus, $k_t$ is well-defined.
If $k_t>0$, then the acceptance criterion \eqref{eq:nonmonotone-line-search} is not satisfied for $k=k_t-1$ and hence $(1-\alpha)\eta^{k_t-1}<\tau$.
Consequently, it holds that $\eta^{k_t}\le\max\big\{1,\frac{\eta\tau}{(1-\alpha)}\big\}$.
\end{proof}

We first establish the global convergence of the NPDNM.

\begin{theorem}\label{thm:global-convergence-2}
Suppose that all assumptions of Lemma \ref{lem:well-def-2} are satisfied and $\{x^t\}_{t=0}^\infty$ is generated by Algorithm \ref{alg:NPDNM} without termination.
We assume further that there exists a set $\mathcal{D}\supset\{x^t\}_{t=0}^\infty$ such that $F$ is uniformly continuous on $\mathcal{D}$ if $M>1$.
Then, $\|x^{t+1}-x^t\|_2$ converges to $0$ and any accumulation point of $\{x^t\}_{t=0}^\infty$ is a d-stationary point of \eqref{problem:general}.
Moreover, it holds that $\lim_{t\to\infty}\dist(x^t, \mathcal{X}^*)=0$ if $\{x^t\}_{t=0}^\infty$ is bounded, where $\mathcal{X}^*$ is the set of d-stationary points of \eqref{problem:general}.
\end{theorem}

\begin{proof}
Let $D_t=\diag(\nabla^2f(x^t))$.
We obtain from Lemma \ref{lem:well-def-2} that
\begin{align}\label{eq:boundedness-hessian-2}
    mI\preceq\eta^{k_t}D_t\preceq\max\left\{1,\frac{\eta\tau}{(1-\alpha)}\right\}L_1I
\end{align}
in the same way as in the proof of Theorem \ref{thm:global-convergence-1}, which implies that the assumptions of Lemma \ref{lem:global-convergence} (see Appendix \ref{sec:lemma}) are satisfied.
This completes the proof.
\end{proof}

Note that $\varepsilon$-stationary point of \eqref{problem:general} can be obtained by a finite number of iterations of Algorithm \ref{alg:NPDNM} with termination criterion \eqref{eq:termination}, similarly to Algorithm \ref{alg:PDNM}.
We further mention that the assumption of the existence of a domain containing $\{x^t\}_{t=0}^\infty$ and on which $F$ is uniformly continuous is not restrictive.
Indeed, since $\{x^t\}_{t=0}^\infty\subset\{x\mid F(x)\le F(x^0)<\infty\}\subset\dom g$, the assumption is satisfied if $F$ is Lipschitz continuous on $\dom g$, or coercive and continuous on $\dom g$.
Although the squared loss function $f(x)=\frac{1}{2}\|b-Ax\|_2^2$ is not uniformly continuous and generally not coercive, the assumption is satisfied for $g$ in the following remark.

\begin{remark}\label{rem:uniform-continuity}
Let $f(x)=\frac{1}{2}\|b-Ax\|_2^2$.
Suppose that $g$ is uniformly continuous on $\dom g$ and bounded below, that is, there exists $l_g\in\mathbb{R}$ such that $g(x)\ge l_g$ for any $x$.
Then, for any $U\in\mathbb{R}$, $F$ is uniformly continuous on $\{x\mid F(x)\le U\}$.
\end{remark}

\begin{proof}
Let $\mathcal{D}=\{x\mid F(x)\le U\}$ and $l_g$ be a lower bound of $g$.
For any $x\in\mathcal{D}$, since it holds that
\begin{align}
    \frac{1}{2}\|b-Ax\|_2^2+l_g\le F(x)\le U,
\end{align}
we have $\|b-Ax\|_2\le\sqrt{2(U-l_g)}$.
It follows from this that
\begin{align}
    \|\nabla f(x)\|_2=\|A^\top(Ax-b)\|_2\le\|A\|_{2,2}\|Ax-b\|_2\le\|A\|_{2,2}\sqrt{2(U-l_g)}
\end{align}
for all $x\in\mathcal{D}$, which implies that $f$ is Lipschitz continuous on $\mathcal{D}$.
Thus, $F$ is uniformly continuous on $\mathcal{D}$.
\end{proof}

For example, the $\ell_1$ norm, the capped $\ell_1$ norm \citep{zhang2010analysis}, and the trimmed $\ell_1$ norm \citep{luo2013new,huang2015two} satisfy the assumption on $g$ in Remark \ref{rem:uniform-continuity}.
These nonsmooth functions are used in the numerical experiments (the proximal mappings of the latter two functions are derived in section \ref{sec:prox-calculus}).

The Q-quadratic convergence property under the complete separability assumption is also valid for the NPDNM.

\begin{theorem}\label{thm:quadratic-convergence-2}
Let $\{x^t\}_{t=0}^\infty$ be a sequence generated by Algorithm \ref{alg:NPDNM} without termination and $f(x)=\sum_{i=1}^n f_i(x_i)$, where $f_i:\mathbb{R}\to\mathbb{R}$.
We suppose that $f$ is $m$-strongly convex, $\nabla f$ and $\nabla^2 f$ are Lipschitz continuous with constant $L_1$ and $L_2$, $g$ is convex, and there exist a set $\mathcal{D}\supset\{x^t\}_{t=0}^\infty$ such that $F$ is uniformly continuous on $\mathcal{D}$, where the last assumption is not necessary when $M=1$.
Then, the sequence $\{x^t\}_{t=0}^\infty$ converges to the unique optimal solution of \eqref{problem:general}.
Furthermore, for sufficiently large $t$, $k_t=0$ holds and the sequence converges Q-quadratically, where $k_t$ is the smallest nonnegative integer $k$ that satisfies \eqref{eq:nonmonotone-line-search}.
\end{theorem}

\begin{proof}
Note that $F(x^t)\le F(x^0)<\infty$.
From the same reasons as in Theorem \ref{thm:quadratic-convergence-1}, the sequence $\{x^t\}_{t=0}^\infty$ converges to the unique optimal solution of \eqref{problem:general}.

Letting
\begin{align}
    y^t=\prox_g^{D_t}\left(x^t-D_t^{-1}\nabla f(x^t)\right),
\end{align}
we have \eqref{eq:cubic-ineq} from the Lipschitz continuity of $\nabla^2f$ , where $D_t=\nabla^2f(x^t)$.
Since the function $g+\frac{1}{2}\|\cdot-(x^t-D_t^{-1}\nabla f(x^t))\|_{D_t}^2$ is $1$-strongly convex with respect to $\|\cdot\|_{D_t}$, it follow from \citet[Theorem 5.25]{beck2017first} that
\begin{align}
    g(x^t)+\frac{1}{2}\|x^t-(x^t-D_t^{-1}\nabla f(x^t))\|_{D_t}^2-g(y^t)-\frac{1}{2}\|y^t-(x^t-D_t^{-1}\nabla f(x^t))\|_{D_t}^2\ge\frac{1}{2}\|y^t-x^t\|_{D_t}^2,
\end{align}
which is equivalent to
\begin{align}
      g(y^t)\le g(x^t)-\nabla f(x^t)^\top(y^t-x^t)-\|y^t-x^t\|_{D_t}^2.
\end{align}
Combining this with \eqref{eq:cubic-ineq}, we have
\begin{align}
    F(y^t) &\le F(x^t)-\frac{1}{2}\|y^t-x^t\|_{D_t}^2\left(1-\frac{L_2}{3m}\|y^t-x^t\|_2\right)\\
    &\le\max_{(t-M+1)_+\le s\le t} F(x^s)-\frac{1}{2}\|y^t-x^t\|_{D_t}^2\left(1-\frac{L_2}{3m}\|y^t-x^t\|_2\right).
\end{align}
Similarity to the proof of Theorem \ref{thm:quadratic-convergence-1}, it holds that \eqref{eq:quadratic} and $\|y^t-x^t\|_2\to0$, and hence the inequality $\alpha\le1-\frac{L_2}{3m}\|y^t-x^t\|_2$ holds for sufficient large $t$.
Therefore, for sufficient large $t$, it holds that $x^{t+1}=y^t$, namely, $k_t=0$.
This and \eqref{eq:quadratic} imply that $\{x^t\}_{t=0}^\infty$ converges quadratically.
\end{proof}

\section{Calculation of scaled proximal mappings of nonconvex functions}\label{sec:prox-calculus}
The bottleneck in the PDNM and NPDNM is the calculation of the diagonal scaled proximal mapping of $g$.
To the best of our knowledge, the computations of scaled proximal mappings for diagonal metrics have not been considered except for those for convex $g$'s by \citet{park2019variable}.
It is easy to see that the diagonal scaled proximal mapping of $g$ can be obtained even if $g$ is nonconvex, as long as $g$ is separable and the proximal mapping of each univariate function can be obtained.
For example, the diagonal scaled proximal mapping of the capped $\ell_1$ norm \citep{zhang2010analysis}, which is defined by
\begin{align}
    g_1(x)\coloneqq\sum_{i\in[n]}\min\{a|x_i|,1\}
\end{align}
with $a>0$, is given as follows.

\begin{proposition}
Let $\lambda>0,~ x\in\mathbb{R}^n$, and $D$ be a diagonal matrix whose elements are positive.
It holds that
\begin{align}\label{eq:separability-prox}
    \prox_{\lambda g_1}^D(x)=\prod_{i\in[n]}\prox_{\lambda\min\{a|\cdot|,1\}}^{d_i}(x_i)
\end{align}
where $d_i$ is the $i$-th diagonal component of $D$,
\begin{align}
    \prox_{\lambda\min\{a|\cdot|,1\}}^{d_i}(x_i)=\begin{cases}
        \{\mathcal{S}_\frac{\lambda a}{d_i}(x_i)\}, &\left(|x_i|\le\frac{\lambda a}{d_i}~ \mbox{and}~ |x_i|<\sqrt{\frac{2\lambda}{d_i}}\right)~ \mbox{or}~ \left(\frac{\lambda a}{d_i}<|x_i|<\frac{\lambda a}{2d_i}+\frac{1}{a}\right),\\
        \{x_i,\mathcal{S}_\frac{\lambda a}{d_i}(x_i)\}, &\left(|x_i|=\sqrt{\frac{2\lambda}{d_i}}\le\frac{\lambda a}{d_i}\right)~ \mbox{or}~ \left(|x_i|=\frac{\lambda a}{2d_i}+\frac{1}{a}>\frac{\lambda a}{d_i}\right),\\
        \{x_i\}, &\left(\sqrt{\frac{2\lambda}{d_i}}<|x_i|\le\frac{\lambda a}{d_i}\right)~ \mbox{or}~ \left(|x_i|>\max\{\frac{\lambda a}{d_i},\frac{\lambda a}{2d_i}+\frac{1}{a}\}\right),
    \end{cases}
\end{align}
and $\mathcal{S}_c(\xi)=\sign(\xi)(|\xi|-c)_+$.
\end{proposition}

\begin{proof}
The equality \eqref{eq:separability-prox} is obvious from the complete separability of $g_1$.
The minimization problem in $\prox_{\lambda\min\{a|\cdot|,1\}}^{d_i}(x_i)$ can be equivalently rewritten as
\begin{align}
    \min_{y\in\mathbb{R}}\left\{\lambda\min\{a|y|,1\}+\frac{d_i}{2}(y-x_i)^2\right\}=\min\bigg\{\underbrace{\min_{y\in\mathbb{R}}\left\{\lambda a|y|+\frac{d_i}{2}(y-x_i)^2\right\}}_{(P_1)},~\underbrace{\min_{y\in\mathbb{R}}\left\{\lambda+\frac{d_i}{2}(y-x_i)^2\right\}}_{(P_2)}\bigg\}.
\end{align}
It is well known that the minimum of $(P_1)$ is given by
\begin{align}
    \min_{y\in\mathbb{R}}\left\{\lambda a|y|+\frac{d_i}{2}(y-x_i)^2\right\}=
    \begin{cases}
        \frac{d_i}{2}x_i^2, &|x_i|\le\frac{\lambda a}{d_i},\\
        \lambda a|x_i|-\frac{(\lambda a)^2}{2d_i}, &|x_i|>\frac{\lambda a}{d_i},
    \end{cases}
\end{align}
which is attained at $\mathcal{S}_\frac{\lambda a}{d_i}(x_i)$.
On the other hand, it is clear that the minimum and minimizer of $(P_2)$ are $\lambda$ and $x_i$, respectively.
We have the desired result by comparing the minima of $(P_1)$ and $(P_2)$.
\end{proof}

Although the trimmed $\ell_1$ norm \citep{luo2013new,huang2015two}, which is defined by
\begin{align}
    g_2(x)\coloneqq\min_{\substack{\Lambda\subset[n]\\|\Lambda|=n-K}}\sum_{i\in\Lambda}|x_i|
\end{align}
with $K\in\{0,1,\ldots,n\}$, is a nonseparable nonconvex function, its diagonal scaled proximal mapping is explicitly obtained.

\begin{proposition}
Let $\lambda>0,~ x\in\mathbb{R}^n$, and $D$ be a diagonal matrix whose elements are positive.
We define $\mathcal{I}$ to be the set consisting of index sets $I$ of size $n-K$ such that $\phi(x_i,d_i)\le\phi(x_j,d_j)$ for any $i\in I,~ j\notin I$, where
\begin{align}
    \phi(x_i,d_i)=
    \begin{cases}
        \frac{d_i}{2}x_i^2, &|x_i|\le\frac{\lambda}{d_i},\\
        \lambda|x_i|-\frac{\lambda^2}{2d_i}, &|x_i|>\frac{\lambda}{d_i}.
    \end{cases}
\end{align}
The set $\prox_{\lambda g_2}^D(x)$ consists of a vector $x^*$ such that
\begin{align}
    x^*_i=
    \begin{cases}
        \mathcal{S}_\frac{\lambda}{d_i}(x_i), &i\in I,\\
        x_i, &i\notin I
    \end{cases}
\end{align}
for some $I\in\mathcal{I}$.
\end{proposition}

\begin{proof}
The minimization problem in $\prox_{\lambda g_2}^D(x)$ can be equivalently rewritten as
\begin{align}
    &\min_{y\in\mathbb{R}^n}\bigg\{\lambda\min_{\substack{\Lambda\subset[n]\\|\Lambda|=n-K}}\sum_{i\in\Lambda}|y_i|+\frac{1}{2}\sum_{i\in[n]}d_i(y_i-x_i)^2\bigg\}\\
    &=\min_{\substack{\Lambda\subset[n]\\|\Lambda|=n-K}}\min_{y\in\mathbb{R}^n}\bigg\{\lambda\sum_{i\in\Lambda}|y_i|+\frac{1}{2}\sum_{i\in[n]}d_i(y_i-x_i)^2\bigg\}\\
    &=\min_{\substack{\Lambda\subset[n]\\|\Lambda|=n-K}}\bigg\{\sum_{i\in\Lambda}\min_{y_i\in\mathbb{R}}\left\{\lambda|y_i|+\frac{d_i}{2}(y_i-x_i)^2\right\}+\sum_{i\notin\Lambda}\min_{y_i\in\mathbb{R}}\left\{\frac{d_i}{2}(y_i-x_i)^2\right\}\bigg\}\\
    &=\min_{\substack{\Lambda\subset[n]\\|\Lambda|=n-K}}\bigg\{\sum_{i\in\Lambda}\phi(x_i,d_i)\bigg\}.
\end{align}
This completes the proof.
\end{proof}

\section{Numerical experiments}\label{sec:experiments}
To illustrate the advantages of our proposed methods, we conducted several numerical experiments.
In Subsection \ref{subsec:quadratic}, numerical results that support Theorem \ref{thm:r-linear-convergence} are provided.
Then, experiments on sparse regression problems using real data are given in Subsections \ref{subsec:regression} and \ref{subsec:logistic}.
The algorithms compared with our proposed algorithms in the experiments are listed below.

\begin{itemize}
    \item \textbf{PGM \citep{parikh2014proximal} with BB rule.}
    The PGM is a typical first-order method for solving problems of the form \eqref{problem:general}.
    A similar line search as in Algorithm \ref{alg:PDNM} is used, and the initial step size is determined by the Barzilai--Borwein (BB) rule \citep{barzilai1988two}.
    \item \textbf{SpaRSA \citep{wright2009sparse}.}
    The SpaRSA applies nonmonotone line search \citep{grippo1986nonmonotone,grippo2002nonmonotone}, which is similar to one as in Algorithm \ref{alg:NPDNM}, and the BB rule to the PGM.
    In \citet{wright2009sparse}, only global convergence for convex optimization problems was discussed, but global convergence property also holds for nonconvex cases \citep{gong2013general}.
    \item \textbf{FISTA \citep{beck2009fast}.}
    The fast iterative shrinkage-thresholding algorithm (FISTA) is an accelerated PGM for convex composite optimization problems proposed by \citet{beck2009fast}.
    \item \textbf{PGM with DBB rule \citep{park2020variable}.}
    The diagonal Barzilai--Borwein (DBB) rule is a method introduced by \citet{park2020variable} to determine the diagonal metric at each iteration without restricting to scalar multiples of the identity matrix.
    A similar nonmonotone line search technique as in the SpaRSA and Algorithm \ref{alg:NPDNM} is applied. 
    Although global convergence was established only for convex cases \citep{park2020variable}, it can be possible for nonconvex cases.
    \item \textbf{ZeroSR1 \citep{becker2012quasi}.}
    The ZeroSR1 is a PNM in which a memoryless symmetric rank-one formula determines the metric.
    For the case where $g$ is a scalar multiple of the $\ell_1$ norm, which is contained in our experiments, its scaled proximal mapping is computed exactly \citep{becker2019quasi}.
    On the other hand, an exact formula of the scaled proximal mapping with symmetric rank-one metric has not been found for nonconvex cases.
\end{itemize}

Algorithms \ref{alg:PDNM} and \ref{alg:NPDNM} were terminated when the terminated criterion \eqref{eq:termination} with $\varepsilon=10^{-12}$ was satisfied or the number of iterations reached $t_{\max}$.
In Subsection \ref{subsec:quadratic}, \ref{subsec:regression}, and \ref{subsec:logistic}, $t_{\max}$ was set to $10^3$, $10^5$, and $10^5$, respectively.
Similar terminated criteria were applied to the compared algorithms.
For the PDNMs, we set $\eta=2,~ \alpha=10^{-2},~ M=5$.
The minimum objective value among all sequences generated by the algorithms is denoted by $F^*$ for each experiment.
For all algorithms, the initial point was set as the origin.
All the algorithms were implemented in MATLAB R2021b, and all the computations were conducted on a PC with OS: Windows CPU: 3.20 GHz and 16.0 GB memory.

\begin{figure}[p]
    \centering
    \includegraphics[width=40em]{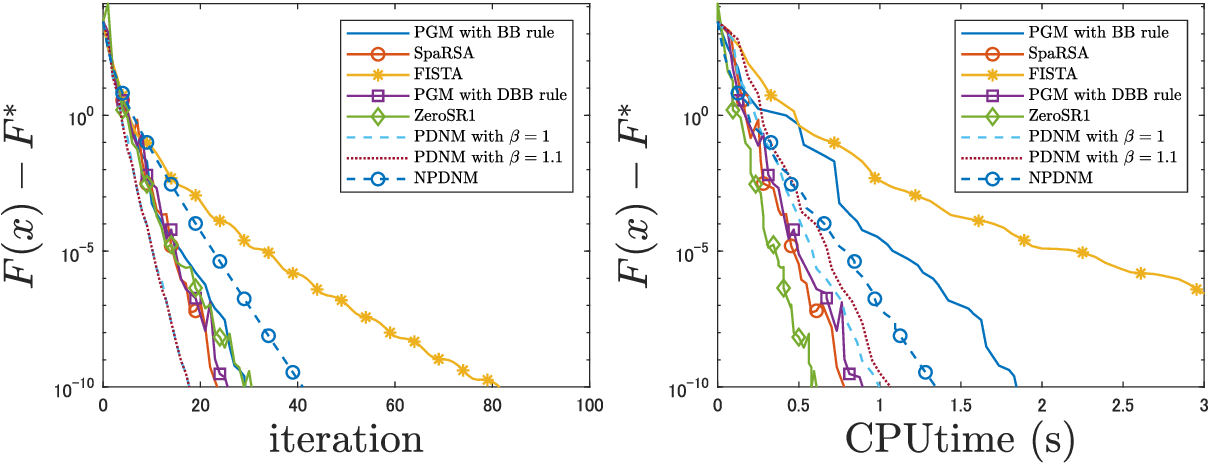}
    \caption{Convergence behaviors of the objective value ($g(x)=\|x\|_1$, $\lambda=0.3$).}
    \label{fig:synthetic-LASSO-03}
    \vspace{20pt}
    \includegraphics[width=40em]{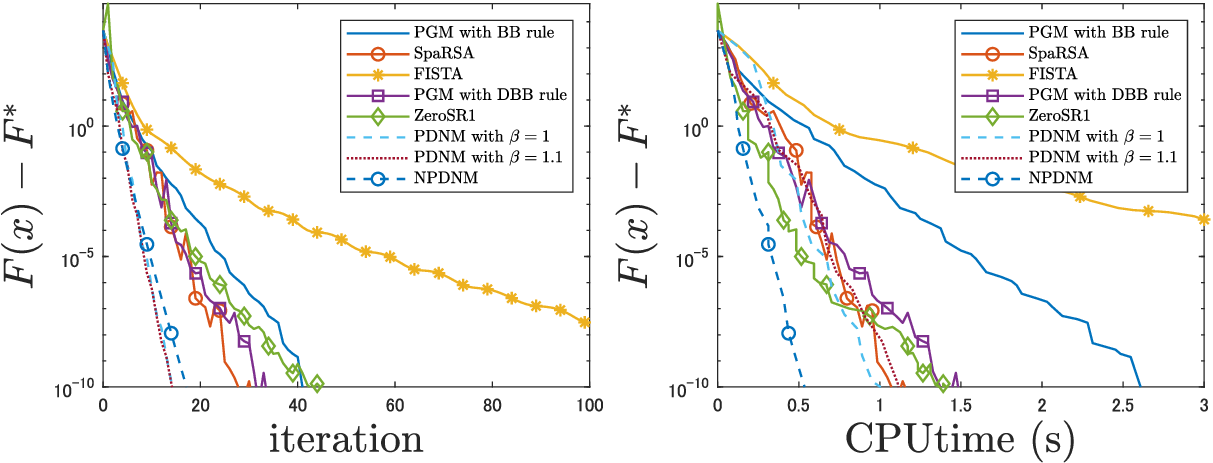}
    \caption{Convergence behaviors of the objective value ($g(x)=\|x\|_1$, $\lambda=0.5$).}
    \label{fig:synthetic-LASSO-05}
    \vspace{20pt}
    \includegraphics[width=40em]{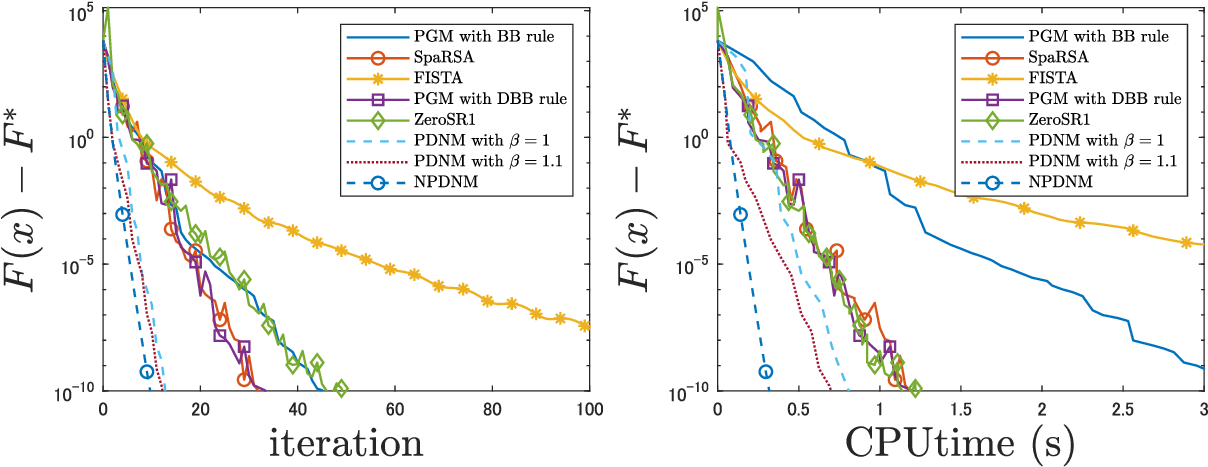}
    \caption{Convergence behaviors of the objective value ($g(x)=\|x\|_1$, $\lambda=0.7$).}
    \label{fig:synthetic-LASSO-07}
\end{figure}

\begin{figure}[p]
    \centering
    \includegraphics[width=40em]{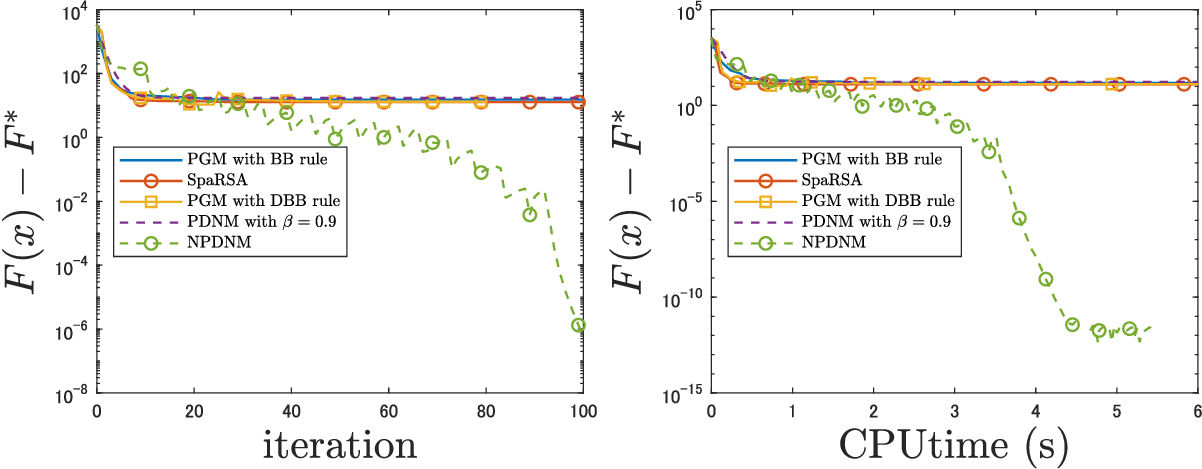}
    \caption{Convergence behaviors of the objective value ($g=g_1$, $\lambda=0.3$).}
    \label{fig:synthetic-CL-03}
    \vspace{20pt}
    \includegraphics[width=40em]{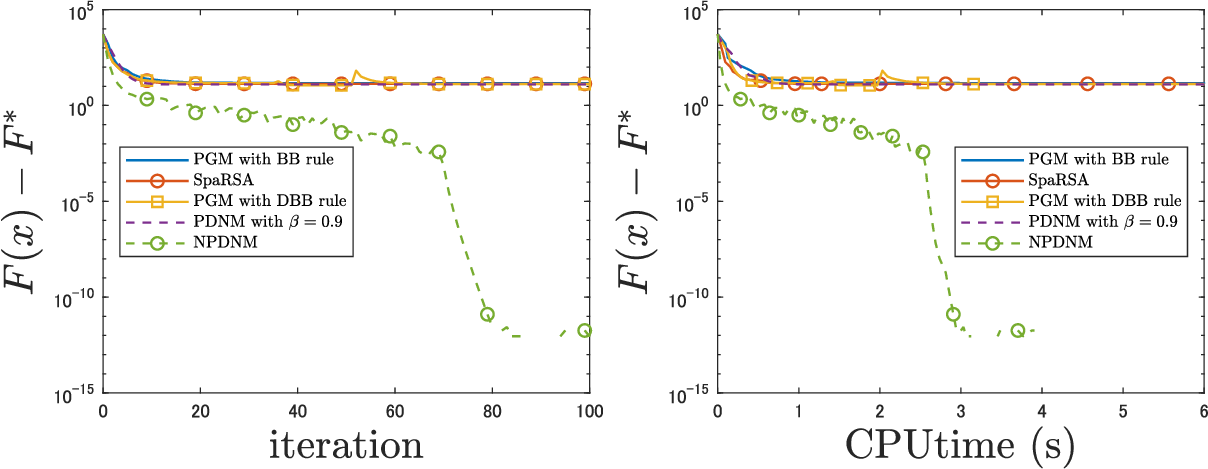}
    \caption{Convergence behaviors of the objective value ($g=g_1$, $\lambda=0.5$).}
    \label{fig:synthetic-CL-05}
    \vspace{20pt}
    \includegraphics[width=40em]{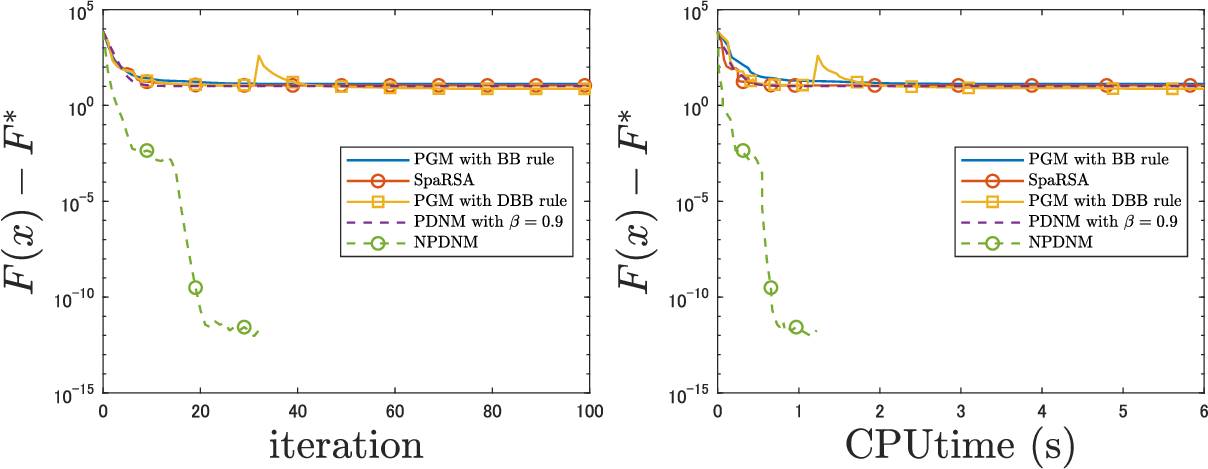}
    \caption{Convergence behaviors of the objective value ($g=g_1$, $\lambda=0.7$).}
    \label{fig:synthetic-CL-07}
\end{figure}

\begin{figure}[p]
    \centering
    \includegraphics[width=40em]{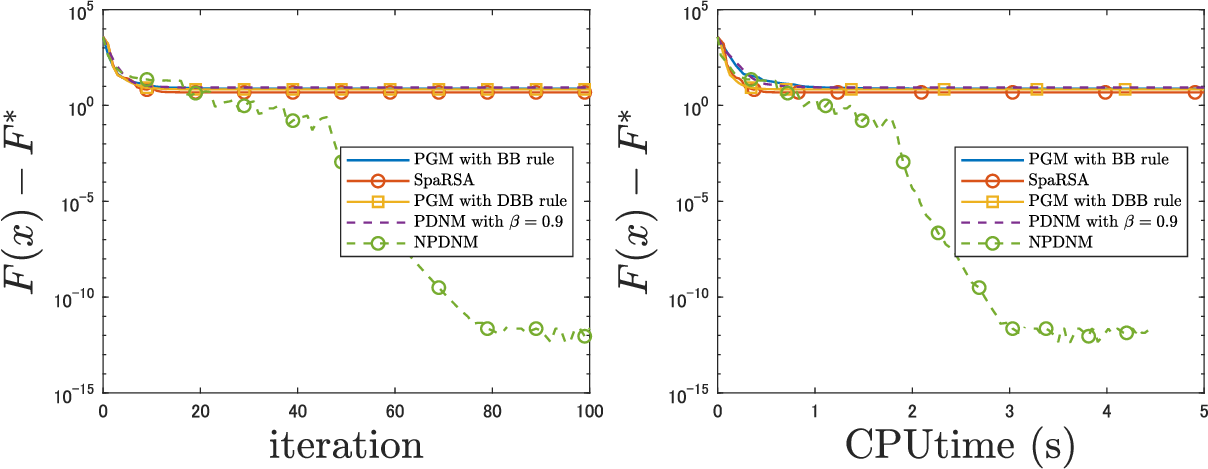}
    \caption{Convergence behaviors of the objective value ($g=g_2$, $\lambda=0.3$).}
    \label{fig:synthetic-TL-03}
    \vspace{20pt}
    \includegraphics[width=40em]{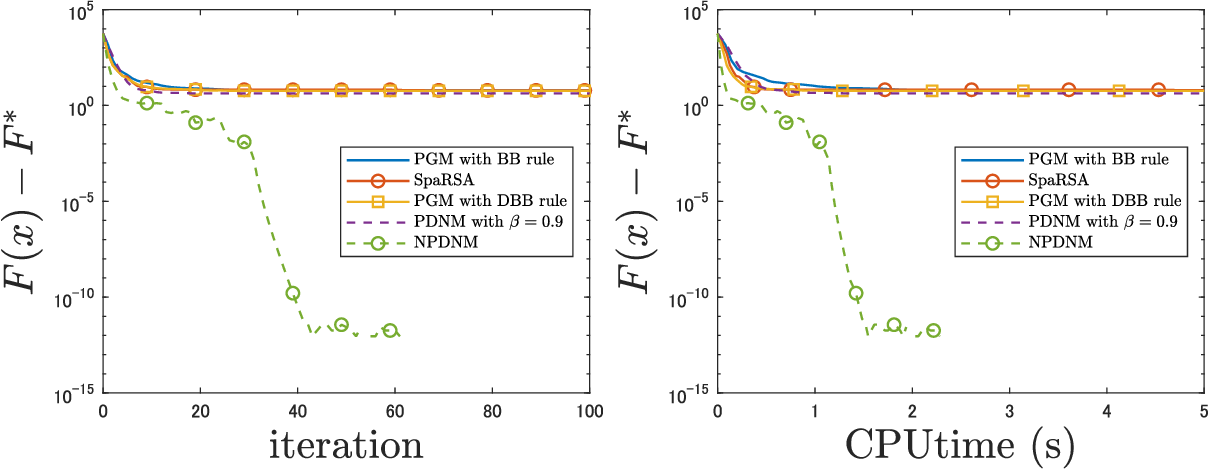}
    \caption{Convergence behaviors of the objective value ($g=g_2$, $\lambda=0.5$).}
    \label{fig:synthetic-TL-05}
    \vspace{20pt}
    \includegraphics[width=40em]{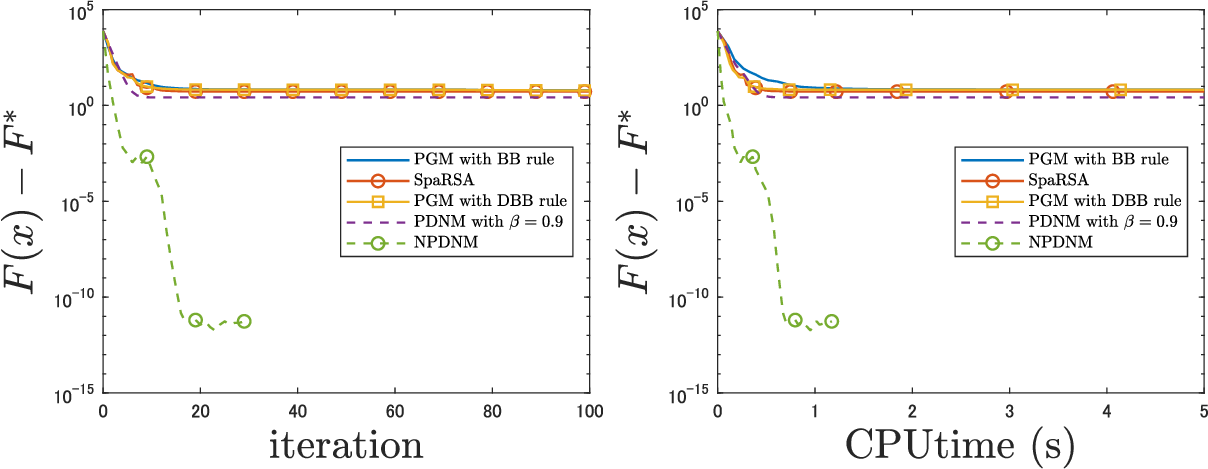}
    \caption{Convergence behaviors of the objective value ($g=g_2$, $\lambda=0.7$).}
    \label{fig:synthetic-TL-07}
\end{figure}

\subsection{Quadratic function with nearly diagonal Hessian}\label{subsec:quadratic}
In this subsection, we consider the problems \eqref{problem:general} where the smooth term $f$ is a strongly convex quadratic function $\frac{1}{2}x^\top Qx+l^\top x$.
Here, $Q=\lambda Q_1+(1-\lambda)Q_2$, $Q_2=\frac{A^\top A}{m}$, and $l=Qe$, where $A$ was a $m\times n$ matrix whose elements were independently drawn from $\mathrm{N}(0,1)$, $Q_1$ was a $n\times n$ diagonal matrix whose diagonal elements were independently drawn from $\mathrm{U}(0,10)$, $e$ was a $n$-dimensional vector whose elements were independently drawn from $\mathrm{N}(0,1)$, $m=n=5000$, and $\lambda=0.3, 0.5, 0.7$.
As $\lambda$ increased, the condition number of $Q$ got worse, like $15.17$, $19.25$, and $32.74$, whereas $Q$ approached a diagonal matrix.

The first experiments shown in Figures \ref{fig:synthetic-LASSO-03}, \ref{fig:synthetic-LASSO-05}, and \ref{fig:synthetic-LASSO-07} dealt with a convex problem with $g(x)=\|x\|_1$, which satisfies the assumptions of Theorem \ref{thm:r-linear-convergence}.
As Theorem \ref{thm:r-linear-convergence} implies, as $Q$ got closer to a diagonal matrix, the performance of the PDNMs outperformed the other methods.
We would like to emphasize that while the convergence of the PGM-type algorithms slowed down as the condition number increased, that of the PDNMs sped up.
The PDNM with $\beta=1.1$ did not satisfy the assumptions of Theorem \ref{thm:r-linear-convergence}, but the performance was not different from the PDNM with $\beta=1$.
We can find that the nonmonotone line search was effective as $Q$ approached a diagonal matrix.

Figures \ref{fig:synthetic-CL-03}-\ref{fig:synthetic-CL-07} and \ref{fig:synthetic-TL-03}-\ref{fig:synthetic-TL-07} show the convergence behaviors in the experiments that solved nonconvex problems where $g$ is the capped $\ell_1$ norm $g_1$ and trimmed $\ell_1$ norm $g_2$, respectively.
Unlike the previous convex case, the PDNM did not outperform the PGM-type algorithms even when $Q$ got close to a diagonal matrix.
However, the PDNM combined with the nonmonotone line search (NPDNM) was able to find better solutions than the other algorithms in each situation.
Moreover, the convergence of the NPDNM sped up as $Q$ approached a diagonal matrix.

\begin{figure}[t]
    \centering
    \includegraphics[width=40em]{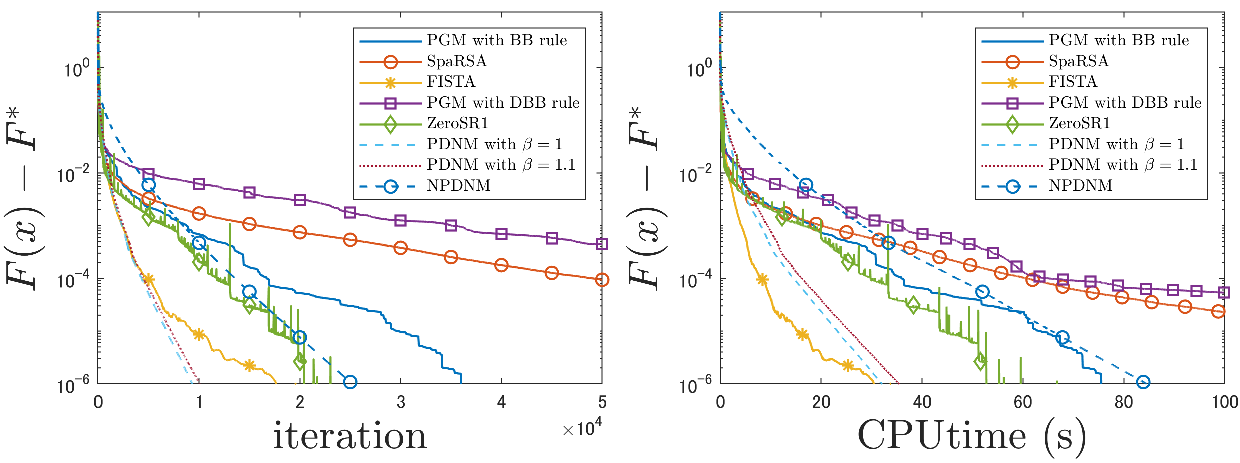}
    \caption{Convergence behaviors of the objective value of sparse regression problem with $g(x)=\frac{1}{m}\|x\|_1$.}
    \label{fig:reg-LASSO}
    \vspace{20pt}
    \includegraphics[width=40em]{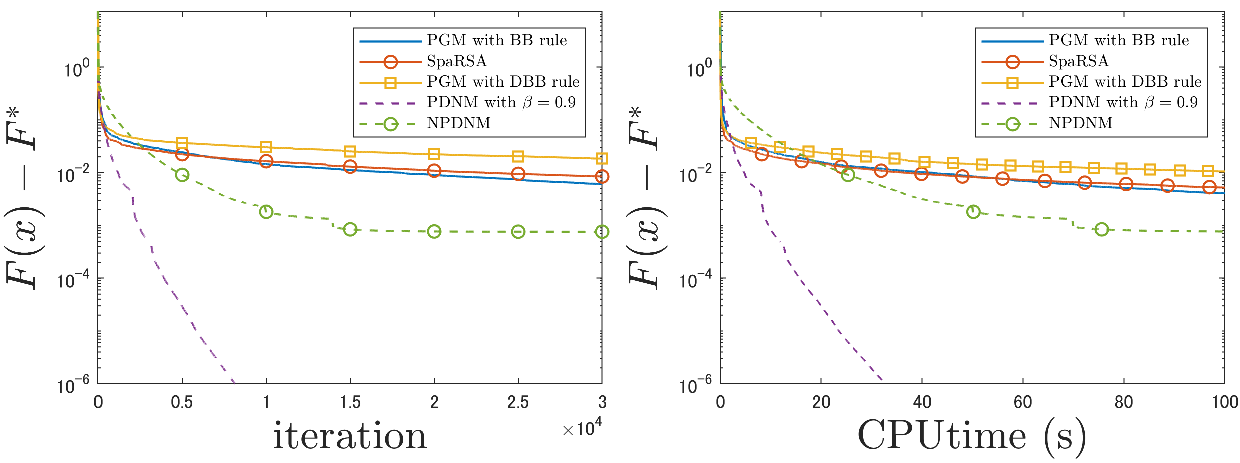}
    \caption{Convergence behaviors of the objective value of sparse regression problem with $g=\frac{1}{m}g_2$.}
    \label{fig:reg-TL}
\end{figure}

\subsection{Sparse regression problem}\label{subsec:regression}
Sparse regression problems, namely, the problems \eqref{problem:general} with $f(x)=\frac{1}{2m}\|b-Ax\|_2^2$ are considered.
Handwritten digit data set obtained from the MATLAB function digitTrain4DArrayData is used.
The sample size $m$ is $5000$ and the number of explanatory variables $n$ is $719$ after removing the meaningless ones.
Since the reciprocal of the condition number of $A^\top A$ is $7.7\times10^{-42}$, the problem is highly ill-conditioned.
We note that $A^\top A$ is never close to a diagonal matrix.
Actually, $\frac{\sigma}{\tau}=4.4\times10^{-42}$.

Figure \ref{fig:reg-LASSO} shows the convergence behaviors in the experiment that solved convex problems where $g$ is the $\ell_1$ norm.
Combining the PDNM with the nonmonotone line search was not effective for this problem.
Since $A^\top A$ is not close to a diagonal matrix, this is consistent with the results of the previous subsection.
The PDNM was the fastest in terms of iterations but fell a little behind the FISTA in terms of CPUtime.

A nonconvex problem where $g$ is the trimmed $\ell_1$ norm $g_2$ was also treated.
The convergence behaviors are shown in Figure \ref{fig:reg-TL}.
The PDNMs succeeded in finding better solutions than the other algorithms.
In particular, the PDNM performed significantly.

\begin{figure}[t]
    \centering
    \includegraphics[width=40em]{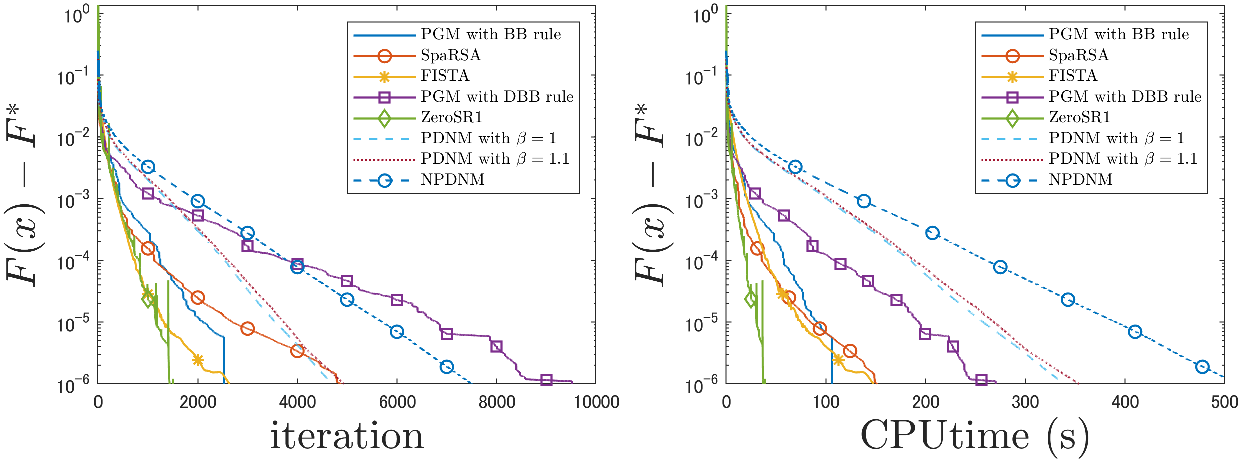}
    \caption{Convergence behaviors of the objective value of sparse logistic regression problem with $g(x)=\frac{1}{m}\|x\|_1$.}
    \label{fig:logis-LASSO}
    \vspace{20pt}
    \includegraphics[width=40em]{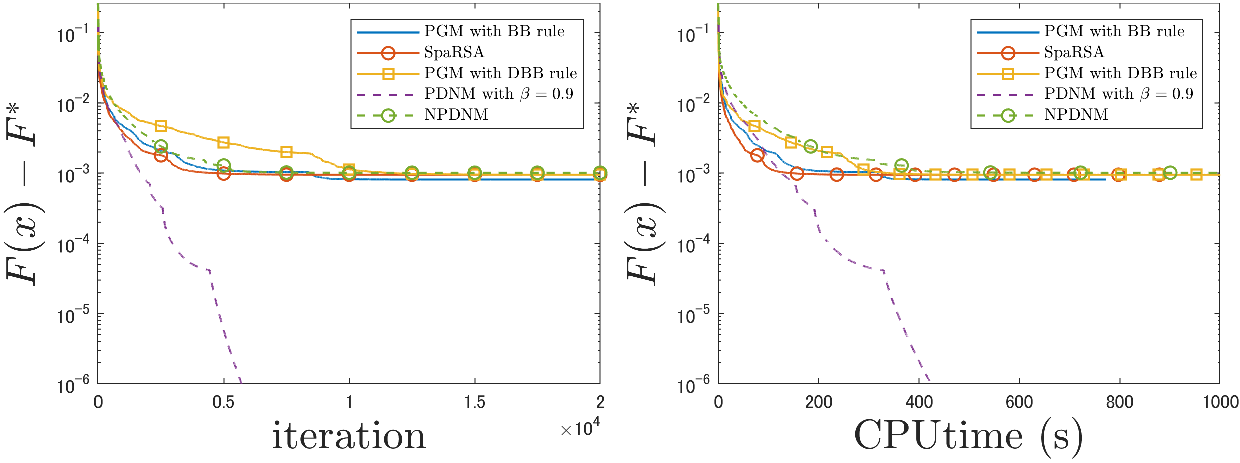}
    \caption{Convergence behaviors of the objective value of sparse logistic regression problem with $g=\frac{1}{m}g_2$.}
    \label{fig:logis-TL}
\end{figure}

\subsection{Sparse logistic regression}\label{subsec:logistic}
The Hessian of $f$ is constant in the previous two subsections.
That is, static metrics are used in the PDNMs.
In this subsection, we consider the following sparse logistic regression problems:
\begin{align}
    \underset{x\in\mathbb{R}^n}{\mbox{minimize}} & \quad \sum_{i\in[m]}\log\left(1+\exp(-b_ia_i^\top x)\right)+\frac{\gamma}{2}\|x\|_2^2+g(x),
\end{align}
where $a_i\in\mathbb{R}^{n}$ is the explanatory variable of the $i$-th sample, $b_i\in\{-1,1\}$ is its label, $g(x)$ is either $\frac{1}{m}\|x\|_1$ or $\frac{1}{m}g_2(x)$, and $\gamma=10^{-2}$.
The Hessian of $f$ is not constant in this setting, namely, variable metrics are used.
We used the same data set as in the previous subsection, with labels of $1$ for samples whose written numbers were $1$, $2$, $4$, or $7$, and $-1$ for those whose numbers were not.

In the convex case, Figure \ref{fig:logis-LASSO} shows that unfortunately, the PDNMs were not effective for this optimization problem.
However, for the nonconvex case, we see from Figure \ref{fig:logis-TL} that the PDNM found a better solution than the other algorithms as in the previous subsection.

\section{Concluding remarks}\label{sec:conclusion}
In this paper, we have proposed PNMs with a new diagonal metric, the PDNM and the NPDNM, based on a simple idea.
Theory suggests that our proposed algorithm PDNM outperforms PGMs in solving convex optimization problems that satisfy certain conditions.
Although there exist numerical experiments suggesting the effectiveness of diagonal metrics, there have been no theoretical results until now.
Our numerical experiments show that the PDNMs are effective algorithms, especially for non-convex problems.
However, no results have yet been obtained that theoretically suggest an advantage of the PDNMs for nonconvex cases.
It is a future task to show this.

\section*{Acknowledgments}
This work was supported in part by JSPS KAKENHI Grant 20K14986 and 23K10999.

\begin{appendices}

\section{Lemma}\label{sec:lemma}
The following lemma is used to show the global convergence results of the PDNMs.

\begin{lemma}\label{lem:global-convergence}
Let $M\ge1,~ m_1, m_2, \alpha>0,~ x^0\in\dom g$ and
\begin{align}
    m_1I\preceq H_t\preceq m_2I,\quad x^{t+1}\in\prox_g^{H_t}(x^t-H_t^{-1}\nabla f(x^t))
\end{align}
for all $t\ge0$.
If it holds that
\begin{align}
    F(x^{t+1})\le\max_{(t-M+1)_+\le s\le t} F(x^s)-\frac{\alpha}{2}\|x^{t+1}-x^t\|_{H_t}^2
\end{align}
and there exist a set $\mathcal{D}\supset\{x^t\}_{t=0}^\infty$ such that $F$ is uniformly continuous on $\mathcal{D}$, then $\|x^{t+1}-x^t\|_2\to0$ and any accumulation point of $\{x^t\}_{t=0}^\infty$ is a d-stationary point of \eqref{problem:general}, where the assumption on uniform continuity is not necessary when $M=1$.

Moreover, it holds that $\lim_{t\to\infty}\dist(x^t, \mathcal{X}^*)=0$ when $\{x^t\}_{t=0}^\infty$ is bounded, where $\mathcal{X}^*$ is the set of d-stationary points of \eqref{problem:general}.
\end{lemma}

\begin{proof}
Note that it holds that $x^t\in\{x\mid F(x)\le F(x^0)<\infty\}$ for any $t\ge0$.
Let $l(t)\in\argmax_{(t-M+1)_+\le s\le t} F(x^s)$, then we obtain that
\begin{align}
    F(x^{l(t+1)}) &=\max_{(t+1-M+1)_+\le s\le t+1} F(x^s)\\
    &=\max\left\{\max_{(t-M+2)_+\le s\le t} F(x^s),F(x^{t+1})\right\}\\
    &\le\max\left\{\max_{(t-M+1)_+\le s\le t} F(x^s),\max_{(t-M+1)_+\le s\le t} F(x^s)-\frac{\alpha}{2}\|x^{t+1}-x^t\|_{H_t}^2\right\}\\
    &=F(x^{l(t)}).
\end{align}
Since $F$ is bounded below, there exists $\overline{F}\in\mathbb{R}$ such that $\{F(x^{l(t)})\}_{t=0}^\infty$ converges to $\overline{F}$.
Suppose that $\lim_{t\to\infty}F(x^{l(t)-j})=\overline{F}$ for $j\ge0$.
It follows from $m_1I\preceq H_{l(t)-j-1}$ that
\begin{align}
    F(x^{l(t)-j}) &\le F(x^{l(l(t)-j-1)})-\frac{\alpha}{2}\|x^{l(t)-j}-x^{l(t)-j-1}\|_{H_{l(t)-j-1}}^2\\
    &\le F(x^{l(l(t)-j-1)})-\frac{\alpha m_1}{2}\|x^{l(t)-j}-x^{l(t)-j-1}\|_2^2,
\end{align}
and hence we have $\|x^{l(t)-j}-x^{l(t)-j-1}\|_2\to0$.
We see from the uniform continuity that
\begin{align}
    \overline{F} &=\lim_{t\to\infty}F(x^{l(t)-j})\\
    &=\lim_{t\to\infty}F(x^{l(t)-j-1}+x^{l(t)-j}-x^{l(t)-j-1})\\
    &=\lim_{t\to\infty}F(x^{l(t)-(j+1)}).
\end{align}
As a result, It holds that $\lim_{t\to\infty}F(x^{l(t)-j})=\overline{F}$ for all $j\ge0$.
By noting that $t-M+1\le l(t)\le t$, we obtain that
\begin{align}\label{eq:convergences}
    \lim_{t\to\infty}F(x^t)=\overline{F},\quad \lim_{t\to\infty}\|x^{t+1}-x^t\|_2=0.
\end{align}
If $M=1$, it follows from $m_1I\preceq H_t$ that
\begin{align}
    F(x^{t+1})\le F(x^t)-\frac{\alpha}{2}\|x^{t+1}-x^t\|_{H_t}^2\le F(x^t)-\frac{\alpha m_1}{2}\|x^{t+1}-x^t\|_2^2.
\end{align}
Since $F$ is bounded below, the equations \eqref{eq:convergences} hold with $\overline{F}=\inf_{t\ge0}F(x^t)\in\mathbb{R}$.
In this case, the uniform continuity is not used.

Let $x^*$ be an accumulation point of $\{x^t\}$ and $\{x^{t_i}\}$ be a subsequence that converges to $x^*$.
Note that $x^*\in\dom g$ because $\{x^t\}\subset\{x\mid F(x)\le F(x^0)\}\subset\dom g$ and $F$ is a closed function.
For any $d\in\mathcal{F}_g(x^*)$, it follows from the optimality of $x^{t_i+1}$ that
\begin{align}
    \nabla f(x^{t_i})^\top x^{t_i+1}+\frac{1}{2}\|x^{t_i+1}-x^{t_i}\|_{H_{t_i}}^2+g(x^{t_i+1})\le\nabla f(x^{t_i})^\top(x^*+\eta d)+\frac{1}{2}\|x^*+\eta d-x^{t_i}\|_{H_{t_i}}^2+g(x^*+\eta d)
\end{align}
for sufficiently small $\eta>0$.
Since it holds that $m_1I\preceq H_t\preceq m_2I$, we see that
\begin{align}
    \nabla f(x^{t_i})^\top x^{t_i+1}+\frac{m_1}{2}\|x^{t_i+1}-x^{t_i}\|_2^2+g(x^{t_i+1})\le\nabla f(x^{t_i})^\top(x^*+\eta d)+\frac{m_2}{2}\|x^*+\eta d-x^{t_i}\|_2^2+g(x^*+\eta d)
\end{align}
by taking the limit $t_i\to\infty$.
From the continuity of $\nabla f$ and the lower semicontinuity of $g$, we obtain that
\begin{align}
    \eta\nabla f(x^*)^\top d+\frac{m_2\eta^2}{2}\|d\|_2^2+g(x^*+\eta d)-g(x^*)\ge0.
\end{align}
Dividing both sides by $\eta$ and taking the limit $\eta\to0$ give
\begin{align}
    F'(x^*;d)=\nabla f(x^*)^\top d+g'(x^*;d)\ge0,
\end{align}
which implies that $x^*$ is a d-stationary point.

Assuming further that $\{x^t\}$ is bounded, we can prove that $\lim_{t\to\infty}\dist(x^t, \mathcal{X}^*)=0$ similarly to the proof of Corollary 1 of \citet{razaviyayn2013unified}.
\end{proof}

\end{appendices}

\bibliography{reference.bib}
\bibliographystyle{plainnat}

\end{document}